\documentclass[aop,MSNbibl,seceqn,dvips]{arximspdf}

\doi{10.1214/14-AOP910} 
\volume{43}
\issue{1}
\pubyear{2015}
\firstpage{405}
\lastpage{434}

\makeatletter

\newcommand{\rright}{\right}
\newcommand{\lleft}{\left}
\newcommand{\rrVert}{\Vert}
\newcommand{\rrvert}{\vert}
\newcommand{\llVert}{\Vert}
\newcommand{\llvert}{\vert}
\newcommand{\G}{\mathcal{G}}
\newcommand{\X}{X}
\newcommand{\1}{\mathbf{1}}
\newcommand{\s}{\sigma}
\renewcommand{\tt}{\mathbf{t}}
\newcommand{\uu}{\mathbf{u}}
\newcommand{\vv}{\mathbf{v}}
\renewcommand{\P}{\mathrm{P}}
\newcommand{\E}{\mathrm{E}}
\newcommand{\R}{\mathbf{R}}
\renewcommand{\d}{\delta}
\renewcommand{\O}{\varnothing}
\newcommand{\F}{\mathcal{F}}
\newcommand{\B}{B}
\renewcommand{\d}{\mathrm{d}}
\newcommand{\ee}{\mathrm{e}}

\newtheorem{proposition}[theorem]{Proposition}
\newtheorem{theorem}{Theorem}[section]
\newtheorem{lemma}[theorem]{Lemma}
\newproclaim{remark}[theorem]{Remark}
\newproclaim{OP}{Open Problems}

\newcommand{\eqref}[1]{(\ref{#1})}
\newcommand{\supp}{\operatorname{supp}}
\renewcommand{\backslash}{\setminus}
\renewcommand{\bar}{\overline}
\renewcommand{\hat}{\widehat}

\def\sfrac#1#2{#1/#2}

\makeatother

\begin{document}
\begin{frontmatter}

\title{Brownian motion and thermal capacity\thanksref{T1}}
\runtitle{Brownian motion and thermal capacity}

\begin{aug}
\author[A]{\fnms{Davar} \snm{Khoshnevisan}\ead[label=e1]{davar@math.utah.edu}\ead[label=u1,url]{http://www.math.utah.edu/\textasciitilde davar}}
\and
\author[B]{\fnms{Yimin} \snm{Xiao}\corref{}\ead[label=e2]{xiao@stt.msu.edu}\ead[label=u2,url]{http://www.stt.msu.edu/\textasciitilde xiaoyimi}}
\runauthor{D. Khoshnevisan and Y. Xiao}
\affiliation{University of Utah and Michigan State University}
\address[A]{Department of Mathematics\\
University of Utah\\
Salt Lake City, Utah 84112-0090\\
USA\\
\printead{e1}\\
\printead{u1}} 
\address[B]{Department of Statistics and Probability\\
Michigan State University \\
C-413 Wells
Hall\\
619 Red Cedar Road\\
East Lansing, Michigan 48824\\
USA\\
\printead{e2}\\
\printead{u2}}
\end{aug}
\thankstext{T1}{Supported in part by NSF Grants
DMS-10-06903 and DMS-13-07470.}

\received{\smonth{10} \syear{2012}}
\revised{\smonth{12} \syear{2013}}

%
\begin{abstract}
Let $W$ denote $d$-dimensional Brownian motion.
We find an explicit formula for the essential supremum\vspace*{1pt}
of Hausdorff dimension
of $W(E)\cap F$, where $E\subset(0,\infty)$
and $F\subset\R^d$ are arbitrary nonrandom compact
sets. Our formula is related intimately to the
thermal capacity of Watson [\textit{Proc. Lond. Math. Soc.} (3) \textbf{37} (1978) 342--362].
We prove also that when $d\ge2$, our formula can be
described in terms of the Hausdorff dimension
of $E\times F$, where $E\times F$
is viewed as a subspace of space time.
\end{abstract}

%

\end{frontmatter}

\section{Introduction}\label{sec1}

Let $W:=\{W(t)\}_{t\ge0}$ denote standard
$d$-dimensional\break Brownian motion where $d\ge1$.
The principal aim of this paper is to describe
the Hausdorff dimension $\dim_{_{\mathrm{H}}}(W(E)\cap F)$ of
the random intersection set $W(E)\cap F$, where
$E$ and $F$ are compact subsets of
$(0,\infty)$ and $\R^d$, respectively.
This endeavor solves what appears to be
an old problem in the folklore of Brownian motion; see M\"orters and
Peres \cite{MoertersPeres}, page 289.

In general, the Hausdorff dimension of $W(E)\cap F$ is
a random variable, and hence we seek only to compute
the $L^\infty(\P)$-norm of that Hausdorff dimension.
The following example---due to Gregory Lawler---highlights the
preceding assertion: Consider $d=1$,
and set $E:=\{1\}\cup[2,3]$ and $F:= [1,2]$. Also, consider the
two events:
%
\begin{eqnarray}
A_1 &:=&\bigl\{ 1\le W(1) \le2, W\bigl([2,3]\bigr)\cap[1,2] =\varnothing\bigr\},\nonumber
\\[-8pt]\\[-8pt]
A_2 &:=& \bigl\{ W(1)\notin[1,2], W\bigl([2,3]\bigr)\subset[1,2]
\bigr\}. \nonumber%
\end{eqnarray}
Evidently, $A_1$ and $A_2$ are disjoint; and each has positive
probability. However, $\dim_{_{\mathrm{H}}}(W(E)\cap F)=0$ on $A_1$,
whereas $\dim_{_{\mathrm{H}}}(W(E)\cap F)=1$ on $A_2$. Therefore,
$\dim_{_{\mathrm{H}}}(W(E)\cap F)$ is nonconstant, as asserted.

Our first result describes our contribution in the case that $d\ge2$.
In order to describe that contribution, let us define $\varrho$ to be
the \emph{parabolic
metric} on \textup{``}space time\textup{''} $\R_+\times\R^d$, that is,
%
\begin{equation}
\varrho \bigl( (s,x); (t,y) \bigr):= \max \bigl( |t-s|^{1/2},
\|x-y\| \bigr).
\end{equation}
The metric space $\mathbf{S}:= (\R_+\times\R^d,\varrho)$ is also called
\emph{space time}, and Hausdorff dimension
of the compact set $E\times F$---viewed as a set in $\mathbf{S}$---is
denoted by $\dim_{_{\mathrm{H}}}(E\times F;\varrho)$. That is,
$\dim_{_{\mathrm{H}}}(E\times F;\varrho)$ is the infimum of $s\ge0$ for which
%
\begin{equation}
\lim_{\varepsilon\to0} \inf \Biggl(\sum_{j=1}^\infty
\bigl|\operatorname{\varrho\mbox{-diam}}(E_j\times F_j)\bigr|^s
\Biggr)<\infty,
\end{equation}
where the infimum is taken
over all closed covers $\{E_j\times F_j\}_{j=1}^\infty$ of $E\times F$
with $\operatorname{\varrho\mbox{-diam}}(E_j\times F_j)< \varepsilon$,
and \textup{``}$\operatorname{\varrho\mbox{-diam}}(\Lambda)$\textup{''} denotes the diameter
of the space--time set $\Lambda$, as measured by the metric $\varrho$.

\begin{theorem}\label{th:dimh:ii}
If $d\ge2$, then
%
\begin{equation}
\label{eq:dimh:ii} \bigl\llVert \dim_{_{\mathrm{H}}} \bigl( W(E)\cap F \bigr) \bigr
\rrVert _{L^\infty
(\P)} = \dim_{_{\mathrm{H}}} (E\times F;\varrho ) -d,
\end{equation}
where \textup{``}$\dim_{_{\mathrm{H}}}A<0$\textup{''} means \textup{``}$A=\O$.\textup{''} Display
\eqref{eq:dimh:ii} continues to hold for $d=1$, provided that \textup{``}$=$\textup{''}
is replaced by \textup{``}$\le$.\textup{''}
\end{theorem}

The following example shows that \eqref{eq:dimh:ii}
does not always hold for $d=1$:
Consider $E:=[0,1]$ and $F:=\{0\}$. Then
a computation on the side shows that
$\dim_{_{\mathrm{H}}}(W(E)\cap F)=0$ a.s.,
whereas $\dim_{_{\mathrm{H}}}(E\times F;\varrho)-d=1$.

On the other hand, Proposition~\ref{pr:positiveLeb} below shows that
if $|F|>0$, where $| \cdot|$ denotes the Lebesgue measure, then
$W(E)\cap F$ shares the properties of the image set $W(E)$.

%
\begin{proposition}\label{pr:positiveLeb}
If $F \subset\R^d$ \textup{(}$d \ge1$\textup{)} is compact and
$|F|>0$, then
%
\begin{equation}
\label{Eq:posiLeb} \bigl\llVert \dim_{_{\mathrm{H}}} \bigl( W(E)\cap F \bigr) \bigr
\rrVert _{L^\infty(\P)} = \min\{d, 2 \dim_{_{\mathrm{H}}}E\}.
\end{equation}
If, in addition, $\dim_{_{\mathrm{H}}}E > \sfrac{1}{2}$ and $d = 1$, then
$\P\{|W(E)\cap F|>0\}>0$.
\end{proposition}

When $F \subset\R^d$ satisfies $|F|>0$, it can be shown that $\dim
_{_{\mathrm{H}}} (E\times F;\varrho )
= 2 \dim_{_{\mathrm{H}}}E +d$. Hence, (\ref{Eq:posiLeb}) coincides with
(\ref{eq:dimh:ii}) when $d \ge2$.
Proposition~\ref{pr:positiveLeb} is proved by showing that, when
$|F|>0$, there exists an explicit \textup{``}smooth\textup{''} random measure on
$W(E)\cap F$.
Thus, the remaining case, and this is the most interesting case,
is when $F$ has Lebesgue measure 0.
The following result gives a suitable (though
quite complicated) formula that is
valid for all dimensions, including $d=1$.

\begin{theorem}\label{th:dimh}
If $F \subset\R^d$ \textup{(}$d \ge1$\textup{)} is compact and
$|F|=0$, then
%
\begin{equation}
\label{eq:dimh} \bigl\llVert \dim_{_{\mathrm{H}}} \bigl( W(E)\cap F \bigr) \bigr
\rrVert _{L^\infty(\P)} = \sup \Bigl\{ \gamma>0\dvtx  \inf_{\mu\in\mathcal{P}_d(E\times F)}
\mathcal{E}_\gamma(\mu)<\infty \Bigr\},
\end{equation}
where $\mathcal{P}_d(E\times F)$ denotes the collection of
all probability measures $\mu$ on $E\times F$ that are
\textup{``}diffuse\textup{''} in the sense that $\mu(\{t\}\times F)=0$
for all $t> 0$, and
%
\begin{equation}
\label{eq:I:gamma} \mathcal{E}_\gamma(\mu):= \int\int\frac{\ee^{-\|x-y\|^2/(2|t-s|)}}{
|t-s|^{d/2} \cdot\|y-x\|^{\gamma}} \mu(\d
s \,\d x) \mu(\d t \,\d y).
\end{equation}
%
\end{theorem}

Theorems \ref{th:dimh:ii} and \ref{th:dimh} are the main results
of this paper. But it seems natural that we also say
a few words about when $W(E)\cap F$ is nonvoid
with positive probability, simply because when $\P\{W(E)\cap F=\O\}=1$
there is no point in computing the Hausdorff dimension of
$W(E)\cap F$.

It is a well-known folklore fact that $W(E)$ intersects $F$ with
positive probability if and only if $E\times F$ has positive thermal
capacity in the sense of Watson \cite{Watson:TC,Watson}. (For a simpler
description, see Proposition~\ref{pr:polar} below.)
This folklore fact can be proved by combining
the results of Doob \cite{Doob} on parabolic potential theory;
specifically, one applies the analytic theory of
\cite{Doob}, Chapter XVII, in the context of space--time Brownian motion
as in \cite{Doob}, Section~13, pages 700--702.
When combined with Theorem~3 of
Taylor and Watson \cite{Taylor:TW}, this folklore fact tells us
the following: If
%
\begin{equation}
\label{eq:TW} \dim_{_{\mathrm{H}}}(E\times F;\varrho)>d,
\end{equation}
then $W(E)\cap F$ is nonvoid with positive probability;
but if $\dim_{_{\mathrm{H}}}(E\times F;\varrho)<d$ then $W(E)\cap F=\O$
almost surely. Kaufman and Wu \cite{Kaufman-Wu} contain
related results. And our Theorem~\ref{th:dimh:ii} states that
the essential supremum of the Hausdorff dimension of $W(E)\cap F$
is the slack in the Taylor--Watson condition \eqref{eq:TW}
for the nontriviality of $W(E)\cap F$.

The proof of Theorem~\ref{th:dimh} yields a simpler interpretation
of the assertion that $E\times F$ has positive thermal capacity,
and relates one of the energy forms that appear in Theorem~\ref{th:dimh},
namely $\mathcal{E}_0$, to the present context. For the sake of completeness,
we state that interpretation
next in the form of Proposition~\ref{pr:polar}. This proposition
provides extra information on the
equilibrium measure---in the sense of parabolic potential theory---for
the thermal capacity of $E \times F$ when $|F|=0$.
[When $|F|>0$, there is nothing to worry about, since
$\P\{W(E)\cap F\neq\O\}>0 $ for every nonempty Borel set
$E \subset(0, \infty)$.]

%
\begin{proposition}\label{pr:polar}
Suppose $F \subset\R^d$ \textup{(}$d \ge1$\textup{)} is compact
and has Lebesgue measure 0.
Then $\P\{W(E)\cap F\neq\O\}>0$ if and only if
there exists a probability measure $\mu\in\mathcal{P}_d(E\times F)$
such that $\mathcal{E}_0(\mu)<\infty$.
\end{proposition}

Theorems \ref{th:dimh:ii} and \ref{th:dimh} both proceed by
checking to see whether or not $W(E)\cap F$ (and a close variant of it)
intersect a sufficiently-thin random set. This so-called \textup{``}codimension
idea\textup{''} was initiated by S. J. Taylor \cite{Taylor:66} and has
been used in other situations as well
\cite{Hawkes:81,Lyons,Peres}. A more detailed account of
the history of stochastic codimension can be found
in the recent book of M\"orters and Peres
\cite{MoertersPeres}, page 287. The broad utility of this method---using
fractal percolation sets as the (thin) testing random sets---was further
illustrated by Yuval Peres \cite{Peres:99}.

Throughout this paper, we adopt the following notation:
For all integers $k\ge1$ and for every $x=(x_1,\ldots,x_k) \in\R^k$,
$\|x\|$ and $|x|$, respectively, define the $\ell^2$ and
$\ell^1$ norms of $x$. That is,
%
\begin{equation}
\|x \|:= \bigl( x_1^2+\cdots+x_k^2
\bigr)^{1/2} \quad \mbox{and}\quad  |x|:= |x_1|+\cdots+|x_k|.
\end{equation}

The rest of the paper is organized as follows. Proposition~\ref{pr:positiveLeb}
is proved in Section~\ref{Sec:posleb}. Then, in Sections~\ref{Sec:Thm-ii} and
\ref{Sec:Thm-main}, Theorems \ref{th:dimh:ii} and \ref{th:dimh} are
proved in
reverse order, since the latter is significantly harder to prove. The main
ingredient for proving Theorem~\ref{th:dimh} is Theorem~\ref
{th:main2} whose
proof is given in Section~\ref{Sec:Thm-main2}. Proposition~\ref
{pr:polar} is
proved in Section~\ref{Subsec:polar}.

\section{Proof of Proposition \texorpdfstring{\protect\ref{pr:positiveLeb}}{1.2}}\label{Sec:posleb}

The upper bound in \eqref{Eq:posiLeb} follows from the well-known fact that
$\dim_{_{\mathrm{H}}}W(E) = \min\{d, 2 \dim_{_{\mathrm{H}}}E\}$ almost
surely. In order to establish
the lower bound in \eqref{Eq:posiLeb}, we first construct a random
measure $\nu$
on $ W(E)\cap F$, and then appeal to a capacity argument. The details follow.

Choose and fix a constant $\gamma$ such that
%
\begin{equation}
0 < \gamma< \min \{d, 2\dim_{_{\mathrm{H}}}E \}.
\end{equation}
According to Frostman's theorem,
there exists a Borel probability measure $\sigma$ on $E$ such that
%
\begin{equation}
\label{Eq:sigma} \int\int\frac{\sigma(\d s) \sigma(\d t)} {|s- t|^{\gamma/2}}< \infty.
\end{equation}

For every integer $n \ge1$, we define a random measure $\mu_n$ on
$E\times F$ via
%
\begin{equation}
\label{Eq:mu1} \int f \,\d\mu_n:= (2\pi n )^{d/2} \int
_{E\times F} f(s, x) \exp \biggl(- \frac{n\|W(s) - x\|^2}{2} \biggr)\sigma(
\d s) \,\d x
\end{equation}
for every Borel measurable function $f\dvtx  E\times F \to\R_+$.
Equivalently,
%
\begin{equation}
\label{Eq:mu1:1} \hspace*{12pt}\int f \,\d\mu_n = \int_{E\times F} \sigma(
\d s) \,\d x f(s, x)\int_{\R^d} \d\xi \exp \biggl(i \bigl\langle
\xi,W(s)-x \bigr\rangle-\frac{\|\xi\|^2}{2n} \biggr),
\end{equation}
thanks to the characteristic function of a Gaussian vector.

Let $\nu_n$ be the image measure of $\mu_n$ under the random mapping
$g\dvtx  E\times F \to\R^d$ defined by $g(s, x):= W(s)$. That is,
$\int\phi\,\d\nu_n:=\int(\phi\circ g) \,\d\mu_n$ for all
Borel-measurable functions $\phi\dvtx \R^d\to\R_+$.
It follows from (\ref{Eq:mu1}) that, if $\{\nu_n\}_{n=1}^\infty$ has
a subsequence
which converges weakly to $\nu$, then $\nu$ is supported on $W(E)
\cap F$.
This $\nu$ will be the desired random measure on $W(E)\cap F$.
Thus, we plan to prove that: (i) $\{\nu_n\}_{n=1}^\infty$ indeed has
a subsequence
which converges weakly; and (ii) use this particular $\nu$ to show that
$\P\{ \dim_{_{\mathrm{H}}}(W(E)\cap F) \ge\gamma\}>0$. This will
demonstrate~\eqref{Eq:posiLeb}.

In order to carry out (i) and (ii), it suffices to verify that there
exist positive and
finite constants $c_1, c_2$ and $c_3$ such that
%
\begin{equation}
\label{Eq:nu1} \E \bigl(\|\nu_n\| \bigr) \ge c_1,\qquad  \E \bigl(\|
\nu_n\|^2 \bigr) \le c_2
\end{equation}
and
%
\begin{equation}
\label{Eq:nu2} \E\int\int\frac{\nu_n(\d x) \nu_n(\d y)}{\|x-y\|^{\gamma}}\le c_3,
\end{equation}
simultaneously for all $n \ge1$, where $\|\nu_n\|:=\nu_n(\R)$
denotes the total mass of $\nu_n$. The rest hinges on a well-known
capacity argument
that is explicitly hashed out in \cite{Kahane85}, pages 204--206; see
also \cite{KX02}, pages 75--76.

It follows from \eqref{Eq:mu1:1} and Fubini's theorem that
%
\begin{eqnarray}
\label{Eq:nu3} %
\E \bigl(\|\nu_n\| \bigr) &=& \int
_{E\times F} \sigma(\d s) \,\d x \int_{\R^d}\d\xi\E
\bigl(\ee^{i \langle\xi, W(s)-x\rangle
} \bigr) \ee^{- \|\xi\|^2/(2n)}\nonumber
\\[-8pt]\\[-8pt]
&= &\int_{E\times F} \biggl(\frac{2 \pi} {s+n^{-1}} \biggr)^{d/2}
\exp \biggl(- \frac{\|x\|^2} {2(s + n^{-1})} \biggr)\sigma(\d s)\, \d x.\nonumber %
\end{eqnarray}
Since $E \subset(0, \infty)$ is compact, we have $\inf E\ge\delta$
for some constant $\delta>0$. Hence, (\ref{Eq:nu3}) implies that
$\inf_{n\ge1}\E(\|\nu_n\|)\ge c_1$ for some constant $c_1>0$,
and this verifies the first inequality in \eqref{Eq:nu1}.
For the second inequality, we use \eqref{Eq:mu1} to see that
%
\begin{equation}
\|\nu_n\|=\|\mu_n\|=(2\pi n)^{d/2}\int
_{E\times F} \exp \biggl(-\frac{n\|W(s)-x\|^2}{2} \biggr)\sigma(\d s)\, \d x.
\end{equation}
We may replace $F$ by all of $\R^d$ in order to find that
$\|\nu_n\|\le(2\pi)^d$ a.s.; whence follows the second inequality in
\eqref{Eq:nu1}.
Similarly, we prove \eqref{Eq:nu2} by writing
\begin{eqnarray*}
&&\int\int\frac{\nu_n(\d x) \nu_n(\d y)}{\|x-y\|^\gamma} \\
&&\qquad =\int_{(E\times F)^2}
\frac{\sigma(\d s) \sigma(\d t)\, \d x \,\d y} {\|
W(t)-W(s)\|^\gamma}
\\
&&\hphantom{\qquad =\int_{(E\times F)^2}}{} \times(2\pi n)^{d} \exp \biggl(-\frac{n\|W(s)-x\|^2-n\|W(t)-y\|
^2}{2} \biggr).
\end{eqnarray*}
We may replace $F$ by $\R^d$, use the scaling property
of $W$ and the fact that $\gamma< d$ in order to see that
\[
\E\int\int\frac{\nu_n(\d x) \nu_n(\d y)}{\|x-y\|^\gamma} \le c \int\int\frac{\sigma(\d s) \sigma(\d t)}{|s-t|^{\gamma/2}}\qquad  \mbox{a.s.}
\]
Therefore, \eqref{Eq:nu2} follows from \eqref{Eq:sigma}.

Finally, we prove the last statement in Proposition~\ref{pr:positiveLeb}.
Since $\dim_{_{\mathrm{H}}}E > \frac{1}2$, Frostman's theorem assures us that
there exists a Borel probability measure $\sigma$ on $E$ such that
(\ref{Eq:sigma})
holds with $\gamma= 1$. We construct a sequence of random measures
$\{\nu_n\}_{n=1}^\infty$ as before, and extract a subsequence
that converges weakly to a random Borel measure $\nu$
on $W(E)\cap F$ such that $\P\{\|\nu\|>0\}>0$.

Let $\widehat{\nu}$ denote the Fourier transform of $\nu$. In accord with
Plancherel's theorem, a sufficient condition for
$\P\{|W(E)\cap F|>0\}>0$ is that
$\widehat{\nu}\in L^2(\R)$.
We apply Fatou's lemma to reduce our problem to the following:
%
\begin{equation}
\label{Eq:nu11} \sup_{n\ge1} \E\int_{-\infty}^\infty
\bigl\llvert \widehat{\nu}_n(\theta)\bigr\rrvert ^2 \,\d
\theta<\infty.
\end{equation}

By \eqref{Eq:mu1:1} and Fubini's theorem,
%
\begin{eqnarray}
\label{Eq:nu12} %
&&\E\int_{-\infty}^\infty\bigl
\llvert \widehat{\nu}_n(\theta)\bigr\rrvert ^2\, \d\theta\nonumber\\
&&\qquad =
\int_{-\infty}^\infty\d\theta\E\int_{\R^2}
\mu_n(\d s \,\d x) \mu_n(\d t \,\d y) \ee^{i \theta(W(s) - W(t))}\nonumber
\\[-8pt]\\[-8pt]
&&\qquad = \int_{-\infty}^\infty\d\theta\int_{(E\times F)^2}
\sigma(\d s) \sigma(\d t) \,\d x \,\d y \int_{\R^2}\d\xi\,\d\eta\nonumber
\\
&&\qquad \quad {}\times\exp \biggl( -i(\xi x +\eta y)- \frac{\xi^2 + \eta^2} {2n} \biggr) \E \bigl(
\ee^{i [(\xi+ \theta) W(s) + (\eta- \theta) W(t)] } \bigr).\nonumber %
\end{eqnarray}
When $0 < s < t$, this last expectation can be written as
%
\begin{equation}
\label{Eq:nu13} \exp \biggl(- \frac{s} 2 ( \xi+\eta)^2 -
\frac{t-s} 2 (\eta- \theta)^2 \biggr).
\end{equation}
By plugging this into (\ref{Eq:nu12}), 
we can write the triple integral in $[\d\theta \,\d\xi \,\d\eta]$ of
\eqref{Eq:nu12} as
%
\begin{eqnarray}
\label{Eq:nu14} %
&&\int_{\R^2} \ee^{- i (\xi x + \eta y)} \exp
\biggl( - \frac{\xi^2 + \eta^2} {2n} -\frac{s} 2 ( \xi+\eta)^2
\biggr) \,\d \xi \,\d\eta\nonumber
\\
&&\quad {} \times\int_{-\infty}^\infty \exp \biggl(-
\frac{t-s} 2 (\eta- \theta)^2 \biggr)\, \d\theta
\\
&&\qquad  = p(x, y) \sqrt{\frac{2 \pi}{t-s}},\nonumber %
\end{eqnarray}
where $p(x, y)$ denotes the joint density function of a
bivariate normal distribution with mean vector $0$ and covariance
matrix $\Gamma^{-1}$,
where
%
\begin{equation}
\Gamma:= \lleft( %
\matrix{ s +n^{-1} & s
\cr
s & s
+n^{-1} } %
\rright).
\end{equation}
%
We plug \eqref{Eq:nu14} into \eqref{Eq:nu12}, replace $F$ by $\R^d$ to
integrate $[\d x \,\d y]$ in order to find that
%
\begin{equation}
\sup_{n\ge1} \E\int_{-\infty}^\infty\bigl
\llvert \widehat{\nu}_n(\theta)\bigr\rrvert ^2 \,\d\theta
\le\textnormal{const}\cdot\int\int\frac{ \sigma(\d s) \sigma(\d
t)}{|s-t|^{1/2}} < \infty.
\end{equation}
This yields (\ref{Eq:nu11}) and completes the proof of
Proposition~\ref{pr:positiveLeb}.

\section{Proof of Theorem \texorpdfstring{\protect\ref{th:dimh}}{1.3}}
\label{Sec:Thm-main}

Here and throughout,
%
\begin{equation}
\B_x(\epsilon):= \bigl\{y\in\R^d\dvtx  \|x-y\|\le\epsilon
\bigr\}
\end{equation}
denotes the radius-$\epsilon$ ball about $x\in\R^d$. Also, define
$\nu_d$ to be the volume of $\B_0(1)$; that is,
%
\begin{equation}
\nu_d:= \frac{2\cdot\pi^{d/2}}{d\Gamma(d/2)}.
\end{equation}

Recall that $\{W(t)\}_{t\ge0}$ denotes a standard Brownian motion in
$\R^d$,
and consider the following \textup{``}parabolic Green function\textup{''}:
For all $t>0$ and $x\in\R^d$,
%
\begin{equation}
\label{eq:heat:kernel} p_t(x):= \frac{\ee^{ -\|x\|^2 / (2t)}}{(2\pi t)^{d/2}}\1_{(0,\infty)}(t).
\end{equation}
The seemingly-innocuous indicator function plays an important role in
the sequel;
this form of the heat kernel appears earlier in
Watson \cite{Watson,Watson:TC} and Doob \cite{Doob}, (4.1), page 266.

As indicated in the \hyperref[sec1]{Introduction}, our proof of Theorem~\ref{th:dimh}
is based
on the codimension argument to check whether or not $W(E)\cap F$
intersect a sufficiently-thin \textup{``}testing\textup{''} random set. One example of such
testing sets could be the range of a stable L\'evy process $X= \{X(t)\}
_{t\ge0}$
in $\R^d$ with index $\alpha\in(0, 2]$. However, this choice of
testing set will only work
for $d \le3$, because the range $X ((0, \infty) )$ will not
be able to intersect $W(E) \cap F$
if $d \ge4$ due to the fact that $X  ((0, \infty) )\cap G
=\varnothing$ a.s. for any
Borel set $G \subset\R^d$ with $\dim_{_{\mathrm{H}}}G < d -\alpha$.

To avoid this restriction and for future applications, we will use
the range of an $N$-parameter additive stable L\'evy process
with index $\alpha$ as the testing set for proving Theorem~\ref{th:dimh}.

Let $X^{(1)},\ldots,X^{(N)}$ be $N$ isotropic stable processes
with common stability index $\alpha\in(0,2]$. We assume that
the $X^{(j)}$'s are totally independent from one another, as well as from
the process $W$, and all take their values in $\R^d$. We
assume also that $X^{(1)},\ldots,X^{(N)}$ have right-continuous
sample paths with left-limits. This assumption can be---and will
be---made without incurring
any real loss in generality. Finally,\vspace*{1pt} our normalization of the processes
$X^{(1)},\ldots,X^{(N)}$ is described as follows:
%
\begin{equation}
\qquad \E \bigl[ \exp \bigl(i \bigl\langle\xi, X^{(k)}(1)\bigr\rangle \bigr)
\bigr] = \ee^{- \|\xi\|^\alpha/2} \qquad \mbox{for all } 1\le k \le N\mbox{ and }\xi\in
\R^d.
\end{equation}

Define the corresponding \emph{additive stable process}
$\X_\alpha:=\{\X_\alpha(\tt)\}_{\tt\in\R^N_+}$ as
%
\begin{equation}
\label{def:X} \X_\alpha(\tt):= \sum_{k=1}^N
X^{(k)} (t_k) \qquad \mbox{for all } \tt:=(t_1,
\ldots,t_N)\in\R^N_+.
\end{equation}
Also, define $\mathcal{C}_\gamma$ to be the capacity corresponding
to the energy form (\ref{eq:I:gamma}). That is, for all compact sets $U
\subset\R_+\times\R^d$ and $\gamma\ge0$,
%
\begin{equation}
\label{CAP:g} \mathcal{C}_\gamma(U):= \Bigl[ \inf_{\mu\in\mathcal{P}_d(U)}
\mathcal{E}_\gamma(\mu) \Bigr]^{-1}.
\end{equation}
%

%
\begin{theorem}\label{th:main2}
If $d> \alpha N$ and $F \subset\R^d$ has Lebesgue measure 0,
then
%
\begin{equation}
\label{Eq:Hitting0} \P \bigl\{ W(E) \cap\X_\alpha\bigl(\R^N_+
\bigr) \cap F \neq\O \bigr\} >0 \quad \Longleftrightarrow\quad  \mathcal{C}_{d-\alpha N} (E
\times F)>0.
\end{equation}
%
\end{theorem}


We can now apply Theorem~\ref{th:main2}
to prove Theorem~\ref{th:dimh}. Theorem~\ref{th:main2}
will be established subsequently.

\begin{pf*}{Proof of Theorem~\ref{th:dimh}}
Suppose $\alpha\in(0,2]$ and $N\in\mathbf{Z}_+$ are chosen
such that $d - \alpha N \in(0, 2)$.
If $\X_\alpha$ denotes an $N$-parameter additive stable
process $\R^d$ whose index is $\alpha\in(0,2]$,
then \cite{KX05}, Theorem~4.4, implies that 
\begin{equation}
\label{eq:codim-ABM} \mbox{ codim } \X_\alpha\bigl(\R^N_+\bigr) =
d-\alpha N.
\end{equation}
This means that $\X_\alpha(\R^N_+)$ will intersect any nonrandom
Borel set $G\subset\R^d\backslash\{0\}$ with $\dim_{_{\mathrm{H}}}(G)>d-\alpha N$,
with positive probability; whereas $\X_\alpha(\R^N_+)$ does
not intersect any $G \subset\R^d\backslash\{0\}$ with
$\dim_{_{\mathrm{H}}}(G)<d-\alpha N$, almost surely.

Define
%
\begin{equation}
\Delta:= \sup \Bigl\{ \gamma>0\dvtx  \inf_{\mu\in\mathcal
{P}_d(E\times F)}
\mathcal{E}_\gamma(\mu)<\infty \Bigr\}
\end{equation}
with the convention that $\sup\varnothing= 0$.

If $\Delta> 0$ and $d-\alpha N<\Delta$, then
$\mathcal{C}_{d-\alpha N}(E\times F)>0$. It follows from
Theorem~\ref{th:main2} and \eqref{eq:codim-ABM} that
%
\begin{equation}
\label{eq:1} \P \bigl\{ \dim_{_{\mathrm{H}}} \bigl( W(E)\cap F \bigr) \ge d-
\alpha N \bigr\}>0.
\end{equation}
Because $d-\alpha N\in(0,\Delta)$ is arbitrary, we have
$\| \dim_{_{\mathrm{H}}}(W(E)\cap F)\|_{L^\infty(\P)} \ge\Delta$.

Similarly, Theorem~\ref{th:main2} and (\ref{eq:codim-ABM}) imply that
%
\begin{equation}
\label{eq:2}\hspace*{24pt} d-\alpha N > \Delta\quad \Longrightarrow\quad \dim_{_{\mathrm{H}}}  \bigl( W(E)\cap
F \bigr) \le d-\alpha N \qquad \mbox{almost surely.}
\end{equation}
Hence, $\|\dim_{_{\mathrm{H}}}(W(E)\cap F)\|_{L^\infty(\P)}\le\Delta$
whenever $\Delta\ge0$.
This proves the theorem.
\end{pf*}

\section{Proof of Theorem \texorpdfstring{\protect\ref{th:main2}}{3.1}}
\label{Sec:Thm-main2}
Our proof of Theorem~\ref{th:main2} is divided into separate parts.
We begin by developing a requisite result in harmonic analysis.
Then we develop some facts about additive L\'evy processes.
After that, we prove Theorem~\ref{th:main2} in two separate parts.

\subsection{Isoperimetry}

Recall that a function $\kappa\dvtx \R^n\to\bar\R_+:= [0, \infty]$ is
\emph{tempered}
if it is measurable and
%
\begin{equation}
\int_{\R^n} \frac{\kappa(x)}{(1+\|x\|)^m} \,\d x<\infty \qquad \mbox{for some $m
\ge0$}.
\end{equation}
A function $\kappa\dvtx \R^n\to\bar\R_+$ is said to be
\emph{positive definite} if it is tempered and
for all rapidly-decreasing test functions
$\phi\dvtx \R^n\to\R$,
%
\begin{equation}
\int_{\R^n}\d x\int_{\R^n}\d y \phi(x)
\kappa(x-y)\phi(y) \ge0.
\end{equation}

Let $\hat{g}$ denote the Fourier transform of a function
(or a measure) $g$. We use the following normalization:
$\hat{g}(\xi)=\int_{\R^n}\exp(i\xi\cdot z)g(z) \,\d z$ when $g\in
L^1(\R^n)$.
We will make heavy use of the following result.

\begin{lemma}\label{lem:FK}
If $\kappa\dvtx \R^n\to\bar\R_+$ is positive definite and lower semicontinuous,
then for all finite Borel measures $\mu$ on $\R^n$,
%
\begin{equation}
\label{eq:FK1} \int\int\kappa(x-y) \mu(\d x) \mu(\d y)= \frac{1}{(2\pi)^d}\int
_{\R^d}\hat\kappa(\xi)\bigl|\hat\mu(\xi)\bigr|^2 \,\d\xi.
\end{equation}
If $\kappa$ is in addition bounded, then in fact
for all finite Borel measures $\mu$ and $\nu$ on~$\R^n$,
%
\begin{equation}
\label{eq:FK2} \int\int\kappa(x-y) \mu(\d x) \nu(\d y)= \frac{1}{(2\pi)^d}\int
_{\R^d}\hat\kappa(\xi)\hat\mu(\xi )\overline{\widehat{\nu}(\xi)}\, \d\xi.
\end{equation}
\end{lemma}

\begin{pf}
Equation \eqref{eq:FK1} is proved in Foondun and Khoshnevisan
\cite{FK}, Corollary~3.4; for a weaker version see \cite{KX}, Theorem~5.2.
We can derive \eqref{eq:FK2} from \eqref{eq:FK1} in a standard way
(\textup{``}polarization\textup{''}): Apply \eqref{eq:FK1} with $\mu+\nu$ in place of
$\mu$ to see that
%
\begin{equation}
\label{eq:FK3} \hspace*{24pt}\int\int\kappa(x-y) (\mu+\nu) (\d x) (\mu+\nu) (\d y)=
\frac{1}{(2\pi)^d}\int_{\R^d}\hat\kappa(\xi)\bigl|(\hat\mu+\hat\nu )
(\xi)\bigr|^2 \,\d\xi.
\end{equation}
Develop both sides, and match the quadratic terms,
using \eqref{eq:FK1}, to finish.
\end{pf}

Lemma~\ref{lem:FK} implies two \textup{``}isoperimetric inequalities\textup{''}
that are stated below as Propositions \ref{pr:isoper} and \ref{pr:isoper:bis}.
Recall that a finite Borel measure $\nu$ on $\R^d$ is said
to be \emph{positive definite} if $\hat\nu(\xi)\ge0$ for
all $\xi\in\R^d$.

%
\begin{proposition}\label{pr:isoper}
Suppose $\kappa\dvtx \R^d\to\bar\R_+$ is a lower semicontinuous
positive-definite function such that $\kappa(x) = \infty$ iff $x=0$.
Suppose $\nu$ and $\sigma$ are two positive definite probability measures
on $\R^d$ that satisfy the following:
\begin{enumerate}[2.]
\item[1.]$\kappa$ and $\kappa*\nu$ are uniformly continuous
on every compact subset of
$\R^d\backslash\{0\}$; and
\item[2.]$(\tau,x)\mapsto(p_\tau*\sigma)(x)$ is uniformly continuous
on every compact subset of $(0, \infty)\times (\R^d \backslash
\{0\} )$.
\end{enumerate}
Then, for all finite Borel measures $\mu$ on $\R_+\times\R^d$,
%
\begin{eqnarray}
&&\int\int(p_{|t-s|}*\sigma) (x-y) (\kappa*\nu) (x-y) \mu(\d t
\,\d x) \mu (\d s \,\d y)\nonumber
\\[-8pt]\\[-8pt]
&&\qquad \le \int\int p_{|t-s|}(x-y) \kappa(x-y) \mu(\d t \,\d x) \mu(\d s \,\d y).\nonumber
\end{eqnarray}
\end{proposition}

\begin{remark}\label{rem:isoper}
The very same proof shows the following slight enhancement:
\textit{Suppose $\kappa$ and $\nu$ are the same as in Proposition~\ref{pr:isoper}.
If $\sigma_1$ and $\sigma_2$ share the properties of $\sigma$
in Proposition~\ref{pr:isoper} and $\hat\sigma_1(\xi)\le\hat
\sigma_2(\xi)$
for all $\xi\in\R^d$, then for all finite Borel measures $\mu$ on
$\R_+\times\R^d$},
%
\begin{eqnarray}
&&\int\int(p_{|t-s|}*\sigma_1) (x-y) (\kappa*\nu)
(x-y) \mu(\d t \,\d x) \mu(\d s \,\d y)\nonumber
\\[-8pt]\\[-8pt]
&&\qquad \le \int\int(p_{|t-s|}*\sigma_2) (x-y) \kappa(x-y) \mu(\d t
\,\d x) \mu(\d s \,\d y).\nonumber %
\end{eqnarray}
Proposition~\ref{pr:isoper} is this in the case that $\sigma
_2:=\delta_0$. An analogous
result holds for positive definite probability measures
$\nu_1$ and $\nu_2$ which satisfy $\hat\nu_1(\xi)\le\hat\nu
_2(\xi)$
for all $\xi\in\R^d$.
\end{remark}

\begin{pf}
Throughout this proof, we choose and fix $\epsilon>0$.

Without loss of generality, we may and will assume that
%
\begin{equation}
\label{eq:HA:wlog} \int\int p_{|t-s|}(x-y) \kappa(x-y) \mu(\d t\, \d x) \mu(\d s\,
\d y)<\infty;
\end{equation}
for there is nothing to prove, otherwise.

Because $p_{|t-s|}$ is positive definite for every nonnegative
$t\neq s$, so are $p_{|t-s|}*\sigma$ and $\kappa*\nu$. Because
$p_{|t-s|}$ is bounded and continuous when $s\neq t$,
it follows from the Bochner--Minlos--Schwartz theorem
that $p_{|t-s|}\times(\kappa*\nu)$ is positive
definite. Therefore, for fixed $t>s$,
Lemma~\ref{lem:FK} applies, and tells us that
for all Borel probability measures $\rho$ on $\R^d$, and for all nonnegative
$t\neq s$,
%
\begin{eqnarray}
\label{Eq:47} %
&&\int\int(p_{|t-s|}*\sigma) (x-y) (\kappa*\nu)
(x-y) \rho(\d x) \rho(\d y)\nonumber
\\[-8pt]\\[-8pt]
&&\qquad =\frac{1}{(2\pi)^d}\int_{\R^d}\d\xi\int_{\R^d}
\d\zeta\ee ^{-(t-s)\|\xi\|^2/2}\hat\sigma(\xi)\hat\kappa(\zeta) \hat\nu(\xi)\bigl|\hat\rho(
\xi-\zeta)\bigr|^2.\nonumber %
\end{eqnarray}
Because the preceding is valid also when $\sigma=\nu=\delta_0$,
and since $0\le\hat\sigma(\xi),\hat\nu(\xi)\le1$ for all
$\xi\in\R^d$, it follows that for all nonnegative $t\neq s$,
%
\begin{eqnarray}
&&\int\int(p_{|t-s|}*\sigma) (x-y) (\kappa*\nu) (x-y) \rho(\d
x) \rho(\d y)\nonumber
\\[-8pt]\\[-8pt]
&&\qquad  \le\int\int p_{|t-s|}(x-y)\kappa(x-y) \rho(\d x) \rho(\d y).\nonumber
\end{eqnarray}
This inequality continues to holds when $\rho$ is a finite Borel measure
on $\R^d$, by scaling. Thus, thanks to Tonelli's theorem,
the proposition is valid whenever
$\mu(\d t \,\d x)=\lambda(\d t)\rho(\d x)$ for two finite Borel measures
$\lambda$ and $\rho$, respectively defined on $\R_+$ and $\R^d$.

Now let us consider a compactly-supported finite measure $\mu$ on
$\R_+\times\R^d$. For all $\eta>0$, define
%
\begin{equation}
\label{eq:HAeta:0} \mathcal{G}(\eta):= \bigl\{ (t,s,x,y)\in(\R_+)^2
\times\bigl(\R ^d\bigr)^2\dvtx  |t-s|\wedge\|x-y\| \ge\eta
\bigr\}.
\end{equation}
It suffices to prove that for all $\eta>0$,
%
\begin{eqnarray}
\label{eq:HAeta} %
&&\int\int_{\mathcal{G}(\eta)} (p_{|t-s|}*
\sigma) (x-y) (\kappa*\nu) (x-y) \mu(\d t \,\d x) \mu(\d s \,\d y)
\nonumber\\[-8pt]\\[-8pt]
&&\qquad \le \int\int_{\mathcal{G}(\eta)} p_{|t-s|}(x-y)\kappa(x-y) \mu(\d t
\,\d x) \mu(\d s \,\d y). \nonumber%
\end{eqnarray}
This is so, because $\kappa(0) = \infty$ and \eqref{eq:HA:wlog} readily
tell us that the product measure $\mu\otimes\mu$ does not charge
%
\begin{equation}
\bigl\{ (t,s,x,y)\in(\R_+)^2\times\bigl(\R^d
\bigr)^2\dvtx  x=y \bigr\};
\end{equation}
and, therefore,
%
\begin{eqnarray}
&&\lim_{\eta\downarrow0} \int\int_{\mathcal{G}(\eta)}
(p_{|t-s|}*\sigma) (x-y) (\kappa*\nu) (x-y) \mu(\d t\, \d x) \mu(\d s\, \d y)
\nonumber\\
&&\qquad = \mathop{\int\int_{s\neq t}}_{\hphantom{\  \ s\neq t}x\neq y} (p_{|t-s|}*
\sigma) (x-y) (\kappa*\nu) (x-y) \mu(\d t\, \d x) \mu(\d s \,\d y)
\\
&&\qquad = \int\int(p_{|t-s|}*\sigma) (x-y) (\kappa*\nu) (x-y) \mu(\d t\, \d x)
\mu(\d s\, \d y).\nonumber %
\end{eqnarray}
And similarly,
%
\begin{eqnarray}
&&\lim_{\eta\downarrow0} \int\int_{\mathcal{G}(\eta)}p_{|t-s|}(x-y)
\kappa(x-y) \mu(\d t \,\d x) \mu(\d s\, \d y)\nonumber
\\[-8pt]\\[-8pt]
&&\qquad = \int\int p_{|t-s|}(x-y)\kappa(x-y) \mu(\d t\, \d x) \mu(\d s\, \d y).\nonumber
\end{eqnarray}
And the proposition follows, subject to \eqref{eq:HAeta}.

Next, we verify \eqref{eq:HAeta} to finish the proof.
One can check directly that
$\mathcal{G}(\eta)\cap\supp(\mu\otimes\mu)$ is compact,
and both mappings $(t,s,x,y)\mapsto(p_{|t-s|}*\sigma)(x-y)\times\allowbreak
(\kappa*\nu)(x-y)$
and $(t,s,x,y)\mapsto p_{|t-s|}(x-y)\kappa(x-y)$
are uniformly continuous on $\mathcal{G}(\eta)\cap\supp(\mu\otimes
\mu)$.

By discretization, we can find finite Borel measures $\{\lambda_j\}
_{j=1}^\infty$---on
$\R_+$---and $\{\rho_j\}_{j=1}^\infty$---on $\R^d$---such that
$\mu$ is the weak limit of $\mu_N:=\sum_{j=1}^N (\lambda_j\otimes
\rho_j)$
as $N\to\infty$. It follows from (\ref{eq:FK2}) and an argument
similar to (\ref{Eq:47})
that for all $\eta>0$ and $N\ge1$,
%
\begin{eqnarray}
&&\int\int_{\mathcal{G}(\eta)} (p_{|t-s|}*\sigma) (x-y) (
\kappa*\nu) (x-y) \mu_N(\d t \,\d x) \mu_N(\d s\, \d y)\nonumber
\\[-8pt]\\[-8pt]
&&\qquad \le \int\int_{\mathcal{G}(\eta)} p_{|t-s|}(x-y) \kappa(x-y)
\mu_N(\d t \,\d x) \mu_N(\d s \,\d y).\nonumber %
\end{eqnarray}
Let $N\uparrow\infty$ to deduce \eqref{eq:HAeta}, and hence
the proposition.
\end{pf}

%
\begin{proposition}\label{pr:isoper:bis}
Suppose $\kappa\dvtx \R\to\bar\R_+$ is a lower semicontinuous
positive-definite function such that $\kappa(x)=\infty$ iff $x=0$.
Suppose $\nu$ and
$\sigma$ are two positive definite probability measures,
respectively on $\R$ and $\R^d$, that satisfy the following:
\begin{enumerate}[2.]
\item[1.]$\kappa$ and $\kappa*\nu$ are uniformly continuous
on every compact subset of
$\R\backslash\{0\}$; and
\item[2.]$(\tau,x)\mapsto(p_\tau*\sigma)(x)$ is uniformly continuous
on every compact subset of $ (0, \infty) \times (\R^d \backslash
\{0\} )$.
\end{enumerate}
Then, for all finite Borel measures $\mu$ on $\R_+\times\R^d$,\vspace*{-1.5pt}
%
\begin{eqnarray}
&&\int\int(p_{|t-s|}*\sigma) (x-y) (\kappa*\nu) (s-t) \mu(\d t
\,\d x) \mu (\d s \,\d y)\nonumber
\\[-8pt]\\[-8pt]
&&\qquad \le \int\int p_{|t-s|}(x-y) \kappa(s-t) \mu(\d t \,\d x) \mu(\d s \,\d y).\nonumber
\end{eqnarray}
\end{proposition}

\begin{pf}
It suffices to prove the proposition in the case that
%
\begin{equation}
\label{eq:mu:lambda:rho} \mu(\d s\, \d x) = \lambda(\d s)\rho(\d x),
\end{equation}
for finite Borel measures $\lambda$ and $\rho$, respectively on
$\R_+$ and $\R^d$. See, for instance, the argument beginning with
\eqref{eq:HAeta:0} in the proof of Proposition~\ref{pr:isoper}.
We shall extend the definition $\lambda$ so that it is a finite Borel measure
on all of $\R$ in the usual way: If $A\subset\R$ is Borel
measurable, then
$\lambda(A):=\lambda(A\cap\R_+)$. This slight abuse in notation should
not cause any confusion in the sequel.

Tonelli's theorem and Lemma~\ref{lem:FK} together imply that
in the case that \eqref{eq:mu:lambda:rho} holds:\vspace*{-1.5pt}
%
\begin{eqnarray}
\label{BIS} &&\int\int(p_{|t-s|}*\sigma) (x-y) (\kappa*\nu) (s-t) \mu(\d
t \,\d x) \mu(\d s\, \d y)
\nonumber
\\
&&\qquad = \int\int\lambda(\d t) \lambda(\d s) (\kappa*\nu) (s-t) \int\int
\rho(\d x) \rho(\d y) (p_{|t-s|}*\sigma) (x-y)
\nonumber
\\[-8pt]\\[-8pt]
&&\qquad =\frac{1}{(2\pi)^d}\int_{\R^d} \hat\sigma(\xi)\bigl|\hat
\rho(\xi )\bigr|^2 \,\d\xi \int\int\lambda(\d t) \lambda(\d s) (\kappa*\nu)
(s-t) \ee^{-|t-s|\cdot\|\xi\|^2/2}
\nonumber
\\
&&\qquad \le\frac{1}{(2\pi)^d}\int_{\R^d} \bigl|\hat\rho(
\xi)\bigr|^2 \,\d\xi \int\int\lambda(\d t) \lambda(\d s) (\kappa*\nu) (s-t)
\ee^{-|t-s|\cdot\|\xi\|^2/2}.\nonumber
\end{eqnarray}
The map $\tau\mapsto\exp\{-|\tau|\cdot\|\xi\|^2/2\}$ is
positive definite on $\R$ for every fixed $\xi\in\R^d$;
in fact, its inverse Fourier transform
is a (scaled) Cauchy density function, which we refer to as
$\vartheta_\xi$.
Therefore, in accord with Lemma~\ref{lem:FK},\vspace*{-1.5pt}
%
\begin{eqnarray}
\label{Eq:BIS2} %
&&\int\int(\kappa*\nu) (s-t) \ee^{-|t-s|\cdot\|\xi\|^2/2}
\lambda(\d t) \lambda(\d s)\nonumber
\\
&&\qquad = \frac{1}{2\pi}\int_{\R}\bigl|\hat\lambda(
\tau)\bigr|^2 (\hat\kappa\hat\nu*\vartheta_\xi) (\tau) \,\d\tau
\le\frac{1}{2\pi}\int_{\R}\bigl|\hat\lambda(
\tau)\bigr|^2 (\hat\kappa*\vartheta_\xi) (\tau) \,\d\tau
\\
&&\qquad =\int\int\kappa(s-t) \ee^{-|t-s|\cdot\|\xi\|^2/2} \lambda(\d t) \lambda(\d s).\nonumber
\end{eqnarray}
The last line follows from the first identity, since we can consider
$\nu=\delta_0$ as
a possibility. Therefore, it follows from \eqref{BIS} and \eqref
{Eq:BIS2} that\vspace*{-1.5pt}
\begin{eqnarray*}
&&\int\int(p_{|t-s|}*\sigma) (x-y) (\kappa*\nu) (s-t) \mu(\d t \,\d x) \mu (
\d s \,\d y)
\\
&&\qquad \le\frac{1}{(2\pi)^d}\int_{\R^d} \bigl|\hat\rho(
\xi)\bigr|^2 \,\d\xi \int\int\lambda(\d t) \lambda(\d s) \kappa(s-t)
\ee^{-|t-s|\cdot\|
\xi\|^2/2}
\\
&&\qquad =\int\int\lambda(\d t) \lambda(\d s) \kappa(s-t) \int\int\rho(\d x) \rho(\d y)
p_{|t-s|}(x-y);
\end{eqnarray*}
the last line follows from the first identity in \eqref{BIS} by considering
the special case that $\nu=\delta_0$ and $\sigma=\delta_0$.
This proves the proposition
in the case that $\mu$ has the form \eqref{eq:mu:lambda:rho}, and the result
follows.
\end{pf}

\subsection{Additive stable processes}
In this subsection, we develop a \textup{``}resolvent density\textup{''} estimate
for the additive stable process $\X_\alpha$.

First of all, note that the characteristic function
$\xi\mapsto\E\exp(i \langle\xi, \X_\alpha(\tt) \rangle)$
of $\X_\alpha(\tt)$ is absolutely integrable
for every $\tt\in\R^N_+\setminus\{\mathbf{0}\}$. Consequently,
the inversion formula applies and tells us that we can always choose
the following as the probability density function of $\X_\alpha(\tt)$:
%
\begin{equation}
\label{eq:g:FT} g_{\tt}(x):= g_{\tt}(\alpha;x) =
\frac{1}{(2\pi)^d}\int_{\R^d}\ee^{-i \langle x, \xi\rangle-
|\tt| \cdot\|\xi\|^\alpha/2} \,\d\xi.
\end{equation}

\begin{lemma}\label{lem:resolvent}
Choose and fix some $\mathbf{a},\mathbf{b}\in(0,\infty)^N$
such that $a_j\le b_j$ for all $1\le j\le N$.
Define
%
\begin{equation}
[\mathbf{a},\mathbf{b}]:= \bigl\{ \mathbf{s}\in\R^N_+\dvtx  a_j\le
s_j\le b_j \mbox{ for all $1\le j\le N$} \bigr\}.
\end{equation}
Then, for all $M>0$
there exists a constant $A_0\in(1,\infty)$---depending only on
the parameters $d$, $N$, $M$, $\alpha$, $\min_{1\le j\le N}a_j$,
and $\max_{1\le j\le N}b_j$---such that for all $x\in[-M,M]^d$,
%
\begin{equation}
\label{Eq:426} A_0^{-1} \le\int_{[\mathbf{a},\mathbf{b}]}
g_{\tt}(x) \,\d\tt\le A_0.
\end{equation}
\end{lemma}

\begin{pf}
Let $\vec{1}:=(1,\ldots,1)\in\R^N$.
Then we may also observe the scaling relation,
%
\begin{equation}
\label{eq:scaling} g_{\tt}(x) = |\tt|^{-d/\alpha}g_{\vec{1}} \biggl(
\frac{x}{|\tt
|^{1/\alpha}} \biggr),
\end{equation}
together with the fact $g_{\vec{1}}$ is an isotropic stable-$\alpha$
density function
on $\R^d$. The upper bound in (\ref{Eq:426}) follows from (\ref{eq:scaling})
and the boundedness of $g_{\vec{1}}(z)$.

On the other hand, since $\mathbf{a}\in(0,\infty)^N$, the lower bound
in (\ref{Eq:426})
follows from (\ref{eq:scaling}) and the well-known fact that $g_{\vec
{1}}(z)$ is
continuous and strictly positive everywhere.
\end{pf}

%
\begin{proposition}\label{pr:resolvent}
Choose and fix some $\mathbf{b}\in(0,\infty)^N$
and define $[\mathbf{0},\mathbf{b}]$ as in Lemma~\ref{lem:resolvent}, and
assume $d > \alpha N$.
Then, for all $M>0$
there exists a constant $A_1\in(1,\infty)$---depending only on
$d$, $N$, $M$, $\alpha$, $\min_{1\le j\le N}b_j$,
and $\max_{1\le j\le N}b_j$---such that for all $x\in[-M,M]^d$,
%
\begin{equation}
\label{Eq:430} \frac{1}{A_1\|x\|^{d-\alpha N}} \le\int_{[\mathbf{0},\mathbf{b}]}
g_{\tt}(x)\, \d\tt\le \frac{A_1}{\|x\|^{d-\alpha N}}.
\end{equation}
\end{proposition}

\begin{pf} Recall the following standard estimate: For all
$R>0$, there exists $C(R)\in(1,\infty)$ and $c(R)\in(0,1)$ such that
%
\begin{equation}
\label{eq:g:bounds} \frac{c(R)}{\|z\|^{d+\alpha}} \le g_{\vec{1}}(z) \le\frac{C(R)}{\|z\|^{d+\alpha}}
\qquad \mbox{for all $z\in\R^d$ with $\|z\|\ge R$}.
\end{equation}
See \cite{Kh}, Proposition~3.3.1, page 380, where this is proved for
$R=2$. The
slightly more general case where $R>0$ is arbitrary is proved in
exactly the same
manner.

Since
%
\begin{equation}
\int_{[\mathbf{0},\mathbf{b}]} g_{\tt}(x)\, \d\tt\le \ee^{|\mathbf{b}|}
\int_{\R^N_+}\ee^{-|\tt|} g_{\tt}(x)\, \d\tt,
\end{equation}
the proof of Proposition~4.1.1 of \cite{Kh}, page 420, shows that the
upper bound
in (\ref{Eq:430}) holds for all $x \in\R^d$.

For the lower bound, we first recall the notation
$\vec{1}:=(1,\ldots,1)\in\R^N$, and then apply \eqref{eq:scaling}
and \eqref{eq:g:bounds} in order to find that
%
\begin{eqnarray}
\label{Eq:432} \int_{[\mathbf{0},\mathbf{b}]} g_{\tt}(x) \,\d
\tt &=&\int_{[\mathbf{0},\mathbf{b}]} |\tt|^{-d/\alpha} g_{\vec{1}} \biggl(
\frac{x}{|\tt|^{1/\alpha}} \biggr)\, \d\tt\nonumber
\\[-8pt]\\[-8pt]
&\ge&\frac{c(1)}{\|x\|^{d+\alpha}} \cdot\hspace*{-8pt}\mathop{\int_{\tt\in[\mathbf{0},\mathbf{b}]:}}_{\hphantom{|\tt|^{1{/}}}|\tt|^{1/\alpha}\le\|x\|}
|\tt| \,\d\tt.\nonumber %
\end{eqnarray}
Clearly, there exists $R_0>0$ sufficiently small such that
whenever $\|x\|\le R_0$,
%
\begin{equation}
\mathop{\int_{\tt\in[\mathbf{0},\mathbf{b}]:}}_{\hphantom{|\tt|^{1{/}}}|\tt|^{1/\alpha}\le\|x\|} |\tt| \,\d\tt\ge\mbox{const}\cdot \|x\|^{\alpha(N+1)},
\end{equation}
and the result follows. On the other hand, if $\|x\|>R_0$,
then the preceding display still holds uniformly for all $x\in[-M,M]^d$.
This proves the proposition.
\end{pf}

We mention also the following; it is an immediate consequence of
Proposition~\ref{pr:resolvent} and the scaling relation \eqref{eq:scaling}.

\begin{lemma}\label{lem:resolvent:bis}
Choose and fix some $\mathbf{b}\in(0,\infty)^N$
and define $[\mathbf{0},\mathbf{b}]$ as in Lemma~\ref{lem:resolvent}.
Then 
there exists a constant $A_2\in(1,\infty)$---depending only on
$d$, $N$, $\alpha$, $\min_{1\le j\le N}b_j$,
and $\max_{1\le j\le N}b_j$---such that for all $x\in\R^d$,
%
\begin{equation}
\label{Eq:compare} \int_{[\mathbf{0},2\mathbf{b}]} g_{\tt}(x) \,\d\tt\le
A_2\int_{[\mathbf{0},\mathbf{b}]} g_{\tt}(x) \,\d\tt.
\end{equation}
\end{lemma}

\begin{pf}
Let $M > 1$ be a constant. If $x \in[-M, M]^d$, then \eqref{Eq:compare}
follows from Proposition~\ref{pr:resolvent}. And if $\|x\| \ge M$,
then \eqref{Eq:compare} holds because
of \eqref{eq:scaling} and \eqref{eq:g:bounds}, together with the
well-known fact
that $g_1$ is continuous and strictly positive everywhere; compare
with the first line in \eqref{Eq:432}.
\end{pf}

\subsection{First part of the proof}\label{subsec:LB}

Our goal, in this first half, is to prove the following:
%
\begin{equation}
\label{first:half} \qquad\mathcal{C}_{d-\alpha N}(E\times F)>0 \quad \Longrightarrow \quad \P \bigl\{W(E)
\cap\X_\alpha\bigl(\R^N_+\bigr)\cap F \neq\O \bigr\} >0.
\end{equation}
By Lemma~4.1 in \cite{KX05}, it is equivalent to prove
%
\begin{equation}
\label{first:half2}\quad \mathcal{C}_{d-\alpha N}(E\times F)>0 \quad \Longrightarrow \quad \E \bigl\{
\lambda_d \bigl( \bigl(W(E)\cap F\bigr) \ominus \X_\alpha
\bigl(\R^N_+\bigr) \bigr) \bigr\} >0,
\end{equation}
where $\lambda_d$ is the Lebesgue measure in $\R^d$ and
$A\ominus B:= \{a-b\dvtx  a \in A, b \in B\}$.

First, let us make some reductions. Because $E\subset(0,\infty)$ and
$F\subset\R^d$
are assumed to be compact, there exists a
$q\in(1,\infty)$ such that
%
\begin{equation}
\label{eq:q} E\subseteq \bigl[ q^{-1}, q \bigr] \quad \mbox{and}\quad  F
\subseteq [ -q, q ]^d.
\end{equation}
We will use $q$ for this purpose unwaiveringly. Notice that if, either
\[
\E \bigl\{\lambda_d \bigl(\X_\alpha\bigl(\R^N_+\bigr) \bigr) \bigr\} >0
\]
 or
there exist some $n \le N-1$
and distinct $i_1, \ldots, i_n \in\{1, \ldots, N\}$ such that
\[
\E \bigl\{ \lambda_d \bigl( \bigl(W(E)\cap F\bigr) \ominus
\X_{i_1, \ldots, i_n}\bigl(\R^n_+\bigr) \bigr) \bigr\} >0,
\]
then (\ref{first:half2}) holds trivially.
In the above, similarly to (\ref{def:X}), $\X_{i_1, \ldots, i_n}$ is
defined by
\[
\X_{i_1 ,\ldots, i_n}(\tt) = \sum_{k=1}^n
X^{(i_k)} (t_{i_k}) \qquad \mbox{for all } \tt:=(t_{i_1},
\ldots,t_{i_n})\in\R^n_+.
\]
Hence, without loss of generality, we can and will assume that
$\E (\X_\alpha(\R^N_+) ) =0$ and $\E \{ \lambda
_d ( (W(E)\cap F)
\ominus \X_{i_1, \ldots, i_n}(\R^n_+) ) \}= 0$ for all $n
\le N-1$
and all distinct $i_1, \ldots, i_n \in\{1, \ldots, N\}$. Since each
L\'evy
process $X_j$ has only countable number of jumps, this assumption implies
%
\begin{equation}
\label{first:half3} \lambda_d \bigl\{ \bigl(W(E)\cap F\bigr) \ominus
\bigl( \overline{\X_\alpha\bigl(\R^N_+\bigr)} \backslash
\X_\alpha \bigl(\R^N_+\bigr) \bigr) \bigr\}=0 \qquad \P\mbox{-a.s.}
\end{equation}

Now we provide some preliminary result for proving (\ref
{first:half2}). Define
%
\begin{equation}
\label{eq:f:phi} f_\epsilon(x):= \frac{1}{\nu_d\epsilon^d}\1_{\B_0(\epsilon)}(x)
\quad \mbox{and} \quad \phi_\epsilon(x):= (f_\epsilon*f_\epsilon) (x).
\end{equation}

For every $\mu\in\mathcal{P}_d(E\times F)$ and
$\epsilon>0$ we define a random variable $Z_\epsilon(\mu)$ by
%
\begin{equation}
\label{eq:I} Z_\epsilon(\mu):= \int_{[1,2]^N}\d\uu \int
_{E\times F}\mu(\d s\, \d x) \phi_\epsilon\bigl(W(s)-x\bigr)
\phi_\epsilon \bigl( \X_\alpha(\uu)-x \bigr).
\end{equation}

%
\begin{lemma}\label{lem:EZ}
There exists a constant $a\in(0,\infty)$ such that
%
\begin{equation}
\inf_{\mu\in\mathcal{P}_d(E\times F)} \inf_{\epsilon\in(0,1)}\E \bigl[
Z_\epsilon(\mu) \bigr] \ge a.
\end{equation}
\end{lemma}

\begin{pf}
Thanks to the triangle inequality,
whenever $u\in\B_0(\epsilon/2)$ and $v\in\B_0(\epsilon/2)$,
we have $u-v\in\B_0(\epsilon)$ and $v\in\B_0(\epsilon)$. Therefore,
for all $u\in\R^d$ and $\epsilon>0$,
%
\begin{eqnarray}
\label{eq:phi:LB} %
\phi_\epsilon(u) &=&\frac{1}{\nu_d^2\epsilon^{2d}} \int
_{\R^d} \1_{\B_0(\epsilon)}(u-v)\1_{\B_0(\epsilon)}(v) \,\d v\nonumber
\\[-8pt]\\[-8pt]
&\ge&\frac{1}{\nu_d^2\epsilon^{2d}} \1_{\B_0(\epsilon/2)}(u)\int_{\R^d}
\1_{\B_0(\epsilon/2)}(v)\, \d v \ge2^{-d} f_{\epsilon/2}(u).\nonumber %
\end{eqnarray}
Because $f_{\epsilon/2}$ is a probability density, and since
$\epsilon\in(0,1)$, the preceding implies that for all $\uu\in[1, 2]^N$
and $x \in\R^d$,
%
\begin{eqnarray}
\label{Eq:435} %
(\phi_\epsilon*g_{\uu}) (x) &=&\int
_{\R^d}\phi_\epsilon(u) g_{\uu}(x-u)\, \d u\nonumber
\\[-8pt]\\[-8pt]
&\ge&2^{-d}\int_{\R^d} f_{\epsilon/2}(u)g_{\uu}(x-u)
\,\d u \ge2^{-d}\inf_{\|z-x\|\le1/2} g_{\uu}(z).\nonumber
\end{eqnarray}
Since $F\subset[-q,q]^d$, (\ref{Eq:435}) and (\ref{eq:scaling})
in the Lemma~\ref{lem:resolvent} tell us that
%
\begin{equation}
a_0:=\inf_{\uu\in[1,2]^N}\inf_{x\in F}\inf
_{\epsilon\in
(0,1)}(\phi_\epsilon*g_{\uu}) (x) >0.
\end{equation}
And, therefore, for all $\epsilon>0$ and $\mu\in\mathcal
{P}_d(E\times F)$,
%
\begin{eqnarray}
\E \bigl[ Z_\epsilon(\mu) \bigr] &=&\int_{E\times F}\mu(\d s
\,\d x) (\phi_\epsilon*p_s) (x)\int_{[1,2]^N} \d
\uu(\phi_\epsilon*g_{\uu}) (x)
\nonumber
\\
&\ge& a_0\int_{E\times F}(
\phi_\epsilon*p_s) (x) \mu(\d s\, \d x) \\
&\ge& a_0\inf
_{s\in[1/q,q]}\inf_{x\in F}\inf_{\epsilon\in
(0,1)}(
\phi_\epsilon*p_s) (x),\nonumber
\end{eqnarray}
which is clearly positive.
\end{pf}

%
\begin{proposition}\label{pr:EZ^2}
There exists a constant $b\in(0,\infty)$ such that the following
inequality holds simultaneously
for all $\mu\in\mathcal{P}_d(E\times F)$:
%
\begin{equation}
\sup_{\epsilon>0} \E \bigl(\bigl\llvert Z_\epsilon(\mu)\bigr
\rrvert ^2 \bigr) \le b\mathcal{E}_{d-\alpha N}(\mu).
\end{equation}
\end{proposition}

\begin{pf}
First of all, let us note the following complement to \eqref{eq:phi:LB}:
%
\begin{equation}
\label{eq:phi:UB} \phi_\epsilon(z) \le2^d f_{2\epsilon}(z)
\qquad \mbox{for all $\epsilon>0$ and $z\in\R^d$}.
\end{equation}

Define, for the sake of notational simplicity,
%
\begin{equation}
\mathcal{Q}_\epsilon(t,x;s,y):= \phi_\epsilon\bigl(W(t)-x\bigr)
\phi _\epsilon\bigl(W(s)-y\bigr).
\end{equation}
Next, we apply the Markov property to find that
for all $(t,x)$ and $(s,y)$ in $E\times F$ such that $s<t$,
and all $\epsilon>0$,
%
\begin{equation}
\E \bigl[\mathcal{Q}_\epsilon(t,x;s,y) \bigr] =\E \bigl[
\phi_\epsilon\bigl( W(s)-y\bigr) \phi_\epsilon\bigl(\tilde
{W}(t-s)+W(s)-x\bigr) \bigr],
\end{equation}
where $\tilde{W}$ is a Brownian motion independent of $W$.
An application of \eqref{eq:phi:UB} yields
%
\begin{eqnarray}
\label{Eq:442} %
\E \bigl[\mathcal{Q}_\epsilon(t,x;s,y)
\bigr] &\le&4^d\E \bigl[f_{2\epsilon}\bigl( W(s)-y\bigr)
f_{2\epsilon}\bigl(\tilde {W}(t-s)+W(s)-x\bigr) \bigr]\nonumber
\\[-8pt]\\[-8pt]
&\le&8^d\E \bigl[f_{2\epsilon}\bigl( W(s)-y\bigr)
f_{4\epsilon}\bigl(\tilde {W}(t-s)-x+y\bigr) \bigr],\nonumber %
\end{eqnarray}
thanks to the triangle inequality. Consequently, we may apply
independence and \eqref{eq:phi:LB} to find that
%
\begin{eqnarray}
\label{Eq:443} %
\E \bigl[\mathcal{Q}_\epsilon(t,x;s,y)
\bigr] &\le&8^d \E \bigl[ f_{2\epsilon}\bigl(W(s)-y\bigr) \bigr]
\cdot \E \bigl[f_{4\epsilon}\bigl(W(t-s)-x+y\bigr) \bigr]\nonumber
\\
&\le&32^d \E \bigl[ \phi_{4\epsilon}\bigl(W(s)-y\bigr) \bigr]
\cdot \E \bigl[\phi_{8\epsilon}\bigl(W(t-s)-x+y\bigr) \bigr]
\\
& =&32^d(\phi_{4\epsilon}*p_s) (y)\cdot(
\phi_{8\epsilon}*p_{t-s}) (x-y).\nonumber %
\end{eqnarray}
Since $s\in E$, it follows that $s\ge1/q$, and hence
$\sup_{z\in\R^d}p_s(z)\le p_{1/q}(0)$. Thus,
%
\begin{equation}
\hspace*{12pt}\E \bigl[\phi_\epsilon\bigl(W(t)-x\bigr)\phi_\epsilon\bigl(W(s)-y
\bigr) \bigr] \le 32^d p_{1/q}(0)\cdot(\phi_{8\epsilon}*p_{t-s})
(x-y).
\end{equation}
By symmetry, the following holds for all $(t,x),(s,y)\in E\times F$
and $\epsilon>0$:
%
\begin{equation}
\label{eq:P1} \hspace*{12pt}\E \bigl[\phi_\epsilon\bigl(W(t)-x\bigr)
\phi_\epsilon\bigl(W(s)-y\bigr) \bigr] \le 32^d
p_{1/q}(0)\cdot(\phi_{8\epsilon}*p_{|t-s|}) (x-y).
\end{equation}
Similarly, we can show that for all $(\uu,x),(\vv,y)\in[1,2]^N\times F$
and $\epsilon>0$:
%
\begin{equation}
\label{eq:G1} \E \bigl[\phi_\epsilon\bigl(\X_\alpha(\uu)-x\bigr)
\phi_\epsilon\bigl(\X _\alpha(\vv)-y\bigr) \bigr] \le
16^d K\cdot(\phi_{8\epsilon}*g_{\uu-\vv}) (x-y),
\end{equation}
where $K:=g_{(1/q,\ldots,1/q)}(0)<\infty$ by \eqref{eq:g:FT},
and the definition of $g_{\tt}(z)$ has\vspace*{1pt} been extended to all $\tt
\in\R^N \setminus\{0\}$ by symmetry, namely,
%
\begin{equation}
\qquad g_{\tt}(z):= |\tt|^{-d/\alpha}g_{\vec{1}} \biggl(
\frac{z}{|\tt
|^{1/\alpha}} \biggr)\qquad  \mbox{for all $z\in\R^d$ and $\tt\in
\R^N\setminus\{0\}$},
\end{equation}
where we recall $\vec{1}:=(1,\ldots,1)\in\R^N$.

To verify (\ref{eq:G1}), we define $Z_1 = \X_\alpha(\uu) - \X
_\alpha(\uu\curlywedge\vv)$
and $Z_2 = \X_\alpha(\vv) - \X_\alpha(\uu\curlywedge\vv)$,
where $\uu\curlywedge\vv
= (u_1\wedge v_1, \ldots, u_N\wedge v_N)$. Then the random variables
$Z_1, Z_2$ and
$ \X_\alpha(\uu\curlywedge\vv)$ are independent. Similarly to
(\ref{Eq:442}) and (\ref{Eq:443}),
the left-hand side of (\ref{eq:G1}) is bounded from above by
%
\begin{equation}
\label{Eq:448} 8^d \E \bigl[ f_{2\epsilon}\bigl(Z_1 +
\X_\alpha(\uu\curlywedge\vv) -x\bigr) f_{4\epsilon}(Z_2-Z_1
+x-y) \bigr].
\end{equation}
By conditional on $Z_1$ and $Z_2$ and applying the unmorality of $\X
_\alpha(\uu\curlywedge\vv)$
(see Remark~2.3 in \cite{KX02}), we see that (\ref{Eq:448}) is at most
%
\begin{eqnarray}
\label{Eq:449} %
&&\frac{8^d}{\nu_d (2 \epsilon)^d} \P \bigl[ \bigl|\X_\alpha(\uu
\curlywedge\vv)\bigr|\le2 \epsilon \bigr] \E \bigl[ f_{4\epsilon}(Z_2-Z_1
+x-y) \bigr]\nonumber
\\[-8pt]\\[-8pt]
&&\qquad\le16^d g_{(1/q,\ldots,1/q)}(0) \cdot(\phi_{8\epsilon}*g_{ \uu
-\vv})
(x-y),\nonumber %
\end{eqnarray}
where we have also use the fact that $Z_2-Z_1$ has density function
$g_{ \uu-\vv}$.
This proves (\ref{eq:G1}).

It follows easily from (\ref{eq:P1}) and (\ref{eq:G1}) that $\E(|
Z_\epsilon(\mu)|^2)$
is bounded from above by a constant multiple of
%
\begin{eqnarray}
&&\int\int(\phi_{8\epsilon}*p_{|t-s|}) (x-y) \biggl(\int
_{[1,2]^{2N}} (\phi_{8\epsilon}*g_{\uu-\vv}) (x-y)\, \d\uu\,\d
\vv \biggr)\nonumber
\\[-8pt]\\[-8pt]
&&\hphantom{\int\int}{}\times\mu(\d t \,\d x) \mu(\d s \,\d y), \nonumber%
\end{eqnarray}
uniformly for all $\epsilon>0$.
Define
%
\begin{equation}
\label{nice:kappa} \kappa(z):= \int_{[0,1]^N} g_{\uu}(z)
\,\d\uu\qquad \mbox{for all $z\in \R^d$}.
\end{equation}
Then we have shown that, uniformly for every $\epsilon>0$,
%
\begin{eqnarray}
\label{jug1} \E \bigl(\bigl\llvert Z_\epsilon(\mu)\bigr\rrvert
^2 \bigr)
&\le&\mbox{const}\cdot \int\int(\phi_{8\epsilon}*p_{|t-s|})
(x-y) (\phi_{8\epsilon}*\kappa) (x-y)\nonumber \\[-8pt]\\[-8pt]
&&\hphantom{\mbox{const}\cdot \int\int}{}\times\mu(\d t \,\d x) \mu(\d s \,\d y).\nonumber
\end{eqnarray}
It follows easily from \eqref{eq:g:FT} that the conditions
of Proposition~\ref{pr:isoper} are met for $\sigma(\d x):=
\nu(\d x):= \phi_{8\epsilon}(x) \,\d x$ and, therefore, that proposition
yields the following bound: Uniformly for all $\epsilon>0$,
%
\begin{equation}
\hspace*{12pt}\E \bigl(\bigl\llvert Z_\epsilon(\mu)\bigr\rrvert ^2 \bigr)
\le\mbox{const}\cdot \int\int p_{|t-s|}(x-y) \kappa(x-y) \mu(\d t \,\d x)
\mu(\d s \,\d y).
\end{equation}
According to Proposition~\ref{pr:resolvent},
$\kappa(z) \le\mbox{ const }/\|z\|^{d-\alpha N}$
uniformly for all $z\in\{x-y\dvtx  x,y\in F\}$, and the proof is thus
completed.
\end{pf}

Now we establish \eqref{first:half}.

\begin{pf*}{Proof of Theorem~\ref{th:main2} (\normalfont{First half})}
If $\mathcal{C}_{d-\alpha N}(E\times F)>0$, then there
exists $\mu_0\in\mathcal{P}_d(E\times F)$ such that
$\mathcal{E}_{d-\alpha N}(\mu_0)<\infty$, by definition.
We apply the Paley--Zygmund inequality \cite{Kh}, page 72, to
Lemma~\ref{lem:EZ} and Proposition~\ref{pr:EZ^2}, with $\mu$ replaced by $\mu_0$,
to find that for all $\epsilon>0$,
%
\begin{equation}
\label{eq:rhs1} \P \bigl\{ Z_\epsilon(\mu_0)>0 \bigr\} \ge
\frac{
\llvert  \E Z_\epsilon(\mu_0)\rrvert ^2}{%
\E (\llvert  Z_\epsilon(\mu_0)\rrvert ^2 )} \ge\frac{a^2/b}{\mathcal{E}_{d-\alpha N}(\mu_0)}.
\end{equation}
If $Z_\epsilon(\mu_0)(\omega)>0$ for some $\omega$ in the underlying
sample space, then it follows from \eqref{eq:I}
and \eqref{eq:phi:LB} that
%
\begin{equation}
\inf_{s\in E}\inf_{x\in F}\inf
_{\uu\in[1,2]^N} \max \bigl( \bigl\|W(s)-x\bigr\|, \bigl\|\X_\alpha(\uu)-x\bigr\|
\bigr) (\omega)\le\epsilon
\end{equation}
for the very same $\omega$. Letting $\epsilon\to0$ in \eqref
{eq:rhs1} we see that,
as the right-most term in \eqref{eq:rhs1} is independent of $\epsilon
>0$, the preceding
establishes
\[
\P \bigl\{ W(E)\cap\overline{\X_\alpha\bigl([a,b]^N\bigr)}
\cap F \neq\O \bigr\}>0.
\]
From the proof of Lemma~4.1 in \cite{KX05}, we see that the above implies
%
\begin{equation}
\label{first:half4} \E \bigl\{ \lambda_d \bigl( \bigl(W(E)\cap F\bigr)
\ominus \overline{\X_\alpha\bigl([a, b]^N\bigr)} \bigr) \bigr
\}>0.
\end{equation}
Because of (\ref{first:half3}), we obtain \eqref{first:half2}. This proves
the first half of the proof of Theorem~\ref{th:main2}.
\end{pf*}

\subsection{Second part of the proof}\label{subsec:UB}

For the second half of our proof, we aim to prove that
%
\begin{equation}
\label{eq:goal2} \quad\P \bigl\{ W(E)\cap\X_\alpha\bigl([a,b]^N
\bigr) \cap F \neq\O \bigr\}>0 \quad \Longrightarrow\quad \mathcal{C}_{d-\alpha N}(E\times F)>0
\end{equation}
for all positive real numbers $a<b$.
This would complete our derivation of Theorem~\ref{th:main2}. In order to
simplify the exposition, we make some reductions. Since $F$ has
Lebesgue measure 0, we may and will assume that $E$ has no isolated points.
Furthermore, we will take $[a, b]^N=
[1,{\sfrac{3}{2}} ]^N$.

Henceforth, we assume that the displayed probability in (\ref{eq:goal2})
is positive. Let $\partial$ be a point that is not in $\R_+\times\R^N_+$,
and we define an $E\times[1,\sfrac{3}{2}]^N \cup\{\partial\}
$-valued random
variable $T = (S, \mathbf{U})$ as follows:
\begin{enumerate}[2.]
\item[1.] If there is no $(s,\uu)\in E\times[1,\sfrac{3}{2}]^N$
such that $W(s)= \X_\alpha(\uu) \in F$, then
$T= (S, \mathbf{U}):= \partial$.
\item[2.] If there exists $(s,\uu)\in E\times[1,\sfrac{3}{2}]^N$
such that $W(s)= \X_\alpha(\uu) \in F$, then we define
$T = (S, \mathbf{U})$ inductively. Let $S$ denote the first time in $E$
when $W$ hits $\X_\alpha( [1,{\sfrac{3}{2}} ]^N) \cap F$, namely,
%
\begin{equation}
\label{Eq:S} S:= \inf \bigl\{s \in E\dvtx  W(s) \in\X_\alpha \bigl(
[1,{ \sfrac{3}{2}} ]^N \bigr) \cap F
\bigr\}.
\end{equation}
It follows from (\ref{Eq:S}) that there is a sequence $(s^n, \uu^n)
\in E\times[1,\sfrac{3}{2}]^N$
such that $s^n \downarrow S$ and $W(s^n) = \X_\alpha (\uu^n
) \in F$ for all $n \ge1$.
Notice that for any subsequence of $\{\uu^n\}$, say $\{\uu^{n_k}\}$,
which converges
to some $\uu= (u_1, \ldots, u_N) \in[1,\sfrac{3}{2}]^N$, we have
$\lim_{k \to\infty} \X_\alpha (\uu^{n_k}  ) = W(S)$.
The limit on the left-hand side
can be expressed as the sum of left or right limits of the L\'evy
processes $X^{(j)}$ at $u_j$
($j=1, \ldots, N$).
For simplicity of notation, we denote this limit by $\overline{\X
}_\alpha(u_1, \ldots, u_N)$.
Then we can define inductively,
%
\begin{eqnarray}
U_1 &:=& \inf %
\bigl\{
u_1 \in[1,\sfrac{3}{2}]\dvtx  \overline{
\X}_\alpha(u_1, u_2, \ldots, u_N) =
W(S)\nonumber\\
&&\hspace*{68pt}\mbox{for some } u_2, \ldots, u_N \in
[1,\sfrac{3}{2}] \bigr\},
\nonumber\\
U_2 &:= &\inf %
\bigl\{ u_2 \in
[1,\sfrac{3}{2}]\dvtx  \overline{\X}_\alpha(U_1, u_2, \ldots, u_N) = W(S)\nonumber\\[-8pt]\\[-8pt]
&&\hspace*{69pt}\mbox{for some }
u_3, \ldots, u_N \in[1,\sfrac{3}{2}]
\bigr\},\nonumber
\\
& \vdots&\nonumber
\\
U_N &:=& \inf \bigl\{ u_N \in[1,
\sfrac{3}{2}]\dvtx  \overline{\X}_\alpha(U_1,
\ldots,U_{N-1}, u_N) = W(S) \bigr\}.\nonumber %
\end{eqnarray}
\end{enumerate}
Note that $\mathbf{U} = (U_1, \ldots, U_N) \in[1,\sfrac{3}{2}]^N$ and
$\overline{\X}_\alpha(\mathbf{U}) = W(S) \in F$ on the event $\{ (S,
\mathbf{U}) \neq\partial\}$.


Now for every two Borel sets $G_1 \subseteq E$ and $G_2 \subseteq F$ we define
%
\begin{equation}
\label{eq:mu-eta} \mu(G_1\times G_2):= \P \bigl\{ S\in
G_1, \overline{\X}_\alpha( \mathbf{U}) \in G_2 | T
\neq\partial \bigr\}.
\end{equation}
Since $\P\{ T \neq\partial\} > 0$, it follows that $\mu$ is a
bona fide probability measure on $E\times F$. Moreover,
$\mu\in\mathcal{P}_d(E\times F)$, since for every $t > 0$,
%
\begin{equation}
\qquad \mu \bigl(\{t\}\times F \bigr) = \P \bigl\{ S = t, \overline{\X
}_\alpha( \mathbf{U}) \in F | T \neq\partial \bigr\} \le\frac{\P\{W(t) \in F\}} {\P\{ T \neq\partial\}} =
0,
\end{equation}
because $F$ has Lebesgue measure 0.

For every $\epsilon>0$, we define $Z_\epsilon(\mu)$ by \eqref{eq:I},
but insist on one (important) change. Namely, now, we use the
Gaussian mollifier,
%
\begin{equation}
\label{Eq:Gk} \phi_\epsilon(z):= \frac{1}{(2\pi\epsilon^2)^{d/2}}\exp \biggl( -
\frac{\|z\|^2}{2\epsilon^2} \biggr),
\end{equation}
in place of $f_\epsilon*f_\epsilon$. (The change in the notation is used
only in this portion of the present proof.)

Thanks to the proof of Lemma~\ref{lem:EZ},
%
\begin{equation}
\label{eq:M1} \inf_{\epsilon\in(0,1)} \E \bigl[ Z_\epsilon(\mu)
\bigr] >0.
\end{equation}
We can argue, as we did in the proof of \eqref{jug1} [e.g., up to a
constant factor,
the inequalities (\ref{eq:P1}) and (\ref{eq:G1}) still hold], to find that
%
\begin{eqnarray}
\quad\quad\sup_{\epsilon\in(0,1)}\E \bigl(\bigl\llvert Z_\epsilon(\mu) \bigr
\rrvert ^2 \bigr)&\le&\mbox{const}\cdot\int\int (\phi_{8\epsilon}*p_{|t-s|})
(x-y) (\phi_{8\epsilon}*\kappa) (x-y)\nonumber \\[-8pt]\\[-8pt]
\quad\quad&&\hphantom{\mbox{const}\cdot\int\int}{}\times\mu(\d s\, \d x) \mu(\d t \,\d y),\nonumber
\end{eqnarray}
where $\kappa$ is defined by \eqref{nice:kappa}. Define
%
\begin{equation}
\label{eq:tildekappa} \tilde\kappa(z):= \int_{[0,1/2]^N}
g_{\tt}(z) \,\d\tt\qquad  \mbox{for all $z\in\R^d$}.
\end{equation}
Thanks to Lemma~\ref{lem:resolvent:bis},
%
\begin{eqnarray}
\quad\quad\sup_{\epsilon\in(0,1)}\E \bigl(\bigl\llvert Z_\epsilon(\mu) \bigr
\rrvert ^2 \bigr)
&\le&\mbox{const}\cdot\int\int (\phi_{8\epsilon}*p_{|t-s|})
(x-y) (\phi_{8\epsilon}*\tilde\kappa) (x-y) \nonumber\\[-8pt]\\[-8pt]
\quad\quad&&\hphantom{\mbox{const}\cdot\int\int}{}\times\mu(\d s \,\d x) \mu(\d t \,\d y).\nonumber
\end{eqnarray}
Now we are ready to explain why we had to change the definition of
$\phi_\epsilon$
from $f_\epsilon*f_\epsilon$ to the present Gaussian ones: In the present
Gaussian case, both subscripts of \textup{``}$8\epsilon$\textup{''} can be replaced by
\textup{``}$\epsilon$\textup{''} at
no extra cost; see \eqref{basic} below. Here is the reason why:

First of all, note that $\phi_\epsilon$ is still positive definite;
in fact,
$\hat\phi_\epsilon(\xi)= \ee^{-\epsilon^2\|\xi\|^2/2}> 0$
for all $\xi\in\R^d$.
Next---and this is important---we can observe that
$\hat\phi_\epsilon\le\hat\phi_\delta$
whenever $0<\delta<\epsilon$. And hence,
the following holds, thanks to Remark~\ref{rem:isoper}:
%
\begin{eqnarray}
\label{basic} \quad\sup_{\epsilon\in(0,1)}\E \bigl(\bigl\llvert
Z_\epsilon(\mu) \bigr\rrvert ^2 \bigr)
&\le&\mbox{const}\cdot\int\int (\phi_\epsilon*p_{|t-s|})
(x-y) (\phi_\epsilon*\tilde\kappa) (x-y) \nonumber\\[-8pt]\\[-8pt]
\quad &&\hphantom{\mbox{const}\cdot\int\int}{}\times\mu(\d s \,\d x) \mu(\d t \,\d y).\nonumber
\end{eqnarray}
This proves the assertion that \textup{``}$8\epsilon$ can be replaced by
$\epsilon$.\textup{''}

Now define a partial order $\prec$ on $\R^N$ as follows:
$\uu\prec\vv$ if and only if $u_i\le v_i$ for all $i=1,\ldots,N$.
Let $\mathcal{X}_{\vv}$ denote the $\s$-algebra generated by the
collection $\{\X_\alpha(\uu)\}_{\uu\prec\vv}$. Also
define $\G:=\{\G_t\}_{t\ge0}$ to be the usual augmented
filtration of the Brownian motion $W$.

According to Theorem~2.3.1 of \cite{Kh}, page 405, $\{\mathcal{X}_{\vv
}\}$
is a commuting $N$-parameter filtration \cite{Kh}, page 233. Hence, so
is the
$(N+1)$-parameter filtration
%
\begin{equation}
\F:= \bigl\{ \F_{s,\uu}; s\ge0, \uu\in\R^N_+ \bigr\},
\end{equation}
where $\F_{s,\uu}:= \G_s\times\mathcal{X}_{\uu}$ is the product
$\sigma$-algebra.

Now, for any fixed $(s,\uu)\in E \times[1,\sfrac{3} {2}]^N$,
%
\begin{equation}
\E \bigl[  Z_\epsilon(\mu) \rrvert \F_{s,\uu} \bigr] \ge
\int_{V(\uu)} \d\vv \mathop{\int_{E \times F}}_{t\ge s}
\mu(\d t\, \d x) \mathcal{T}_\epsilon(t,x;\vv),
\end{equation}
where
%
\begin{equation}
V(\uu):= \bigl\{ \vv\in[1,2]^N\dvtx  u_j\le
v_j \mbox{ for all $1\le j\le N$} \bigr\}
\end{equation}
and
%
\begin{equation}
\mathcal{T}_\epsilon(t,x;\vv):=\E \bigl[ \phi_\epsilon
\bigl(W(t) -x\bigr) \phi_\epsilon\bigl(\X_\alpha(\vv) -x\bigr) |
\F_{s,\uu} \bigr].
\end{equation}
Thanks to independence, and the respective Markov properties of
the processes $W, X^{(1)},\ldots,X^{(N)}$,
%
\begin{eqnarray}
\mathcal{T}_\epsilon(t,x;\vv) &=&\E \bigl[
\phi_\epsilon\bigl(W(t) -x\bigr) | \mathcal{G}_s \bigr]\cdot\E
\bigl[ \phi_\epsilon\bigl(\X _\alpha(\vv) -x\bigr) |
\mathcal{X}_{\uu} \bigr]\nonumber
\\[-8pt]\\[-8pt]
&=& (\phi_\epsilon*p_{t-s}) \bigl(x-W(s) \bigr)\cdot(
\phi_\epsilon*g_{\vv-\uu
}) \bigl(x-\X_\alpha(\uu)\bigr).\nonumber
\end{eqnarray}
Therefore, the definition \eqref{eq:tildekappa} of $\tilde\kappa$
and the triangle inequality together reveal that with probability one,
%
\begin{eqnarray}
\label{eq:C} %
&&\E \bigl[  Z_\epsilon(\mu) \rrvert
\F_{s,\uu} \bigr]\nonumber\\
&&\qquad  \ge\1_{\{(S, \mathbf{U})\ne\partial\}}(\omega)
\\
&&\qquad \quad {} \times\mathop{\int_{E \times F }}_{t > s} (
\phi_\epsilon*p_{t-s}) \bigl(x-W(s)\bigr) (\phi_\epsilon*
\tilde\kappa) \bigl(x-\X_\alpha(\uu)\bigr) \mu(\d t \,\d x).\nonumber %
\end{eqnarray}
This inequality is valid almost
surely, simultaneously for all $s$ in a dense countable subset of $E$ (which
will be assumed as a subset of $\mathbf{Q}_+$ for simplicity of
notation) and all
$\uu\in[1,3/2]^N\cap\mathbf{Q}^N_+$.

Select points with rational coordinates that
converge, coordinatewise from the above and below, to $(S(\omega),\mathbf
{U}(\omega))$.
In this way, we find that
%
\begin{eqnarray}
\label{eq:D} %
&&\mathop{\sup_{s\in E,\mathbf{u}\in[1,\sfrac{3} {2}]^N }}_{
\mathrm{all\ rational\ coords}}
\E \bigl[  Z_\epsilon(\mu) \rrvert \F_{s,\uu} \bigr]\nonumber \\
&&\qquad \ge
\1_{\{(S, \mathbf{U})\ne\partial\}
}(\omega)
\\
&&\qquad \quad {} \times\mathop{\int_{E\times F}}_{t > S} (
\phi_\epsilon*p_{t-S}) \bigl(x-W(S)\bigr) (\phi_\epsilon*
\tilde\kappa) \bigl(x- \overline{\X}_\alpha(\mathbf{U})\bigr) \mu(\d t \,\d x).\nonumber
\end{eqnarray}
This is valid $\omega$ by $\omega$. We square both sides of \eqref{eq:D}
and then apply expectations to both sides in order to obtain the
following:
%
\begin{eqnarray}
\label{eq:SS} %
&&\E \Bigl\{ \Bigl( \sup_{(s,\uu)\in\mathbf{Q}^{N+1}_+} \E \bigl[
 Z_\epsilon(\mu) \rrvert \F_{s,\uu} \bigr]
\Bigr)^2 \Bigr\}\nonumber \\
&&\qquad \ge\P \bigl\{ (S, \mathbf{U}) \neq\partial \bigr\}
\\
&&\qquad \quad {}\times \E \biggl[  \biggl( \mathop{\int_{E\times F}}_{t > S}
\Psi_\epsilon(t,x) \mu(\d t\, \d x) \biggr)^2 \rrvert (S,
\mathbf{U}) \neq\partial \biggr],\nonumber %
\end{eqnarray}
where
\[
\Psi_\epsilon(t,x):=(\phi_\epsilon*p_{t-S})
\bigl(x-W(S)\bigr) (\phi_\epsilon*\tilde\kappa) \bigl(x- \overline{
\X}_\alpha(\mathbf{U})\bigr).
\]

According to \eqref{eq:mu-eta}, and because $W(S) = \overline{ \X
}_\alpha(\mathbf{U})$ on
$\{ (S, \mathbf{U}) \neq\partial\}$, the conditional expectation in
\eqref{eq:SS}
is equal to the following:
%
\begin{equation}
\label{5:37a} \int \biggl( \mathop{\int_{E\times F}}_{t > s}
(\phi_\epsilon*p_{t-s}) (x-y) (\phi_\epsilon*\tilde
\kappa) (x-y) \mu(\d t\, \d x) \biggr)^2 \mu(\d s\, \d y).
\end{equation}
%
In view of the Cauchy--Schwarz inequality, the quantity in
\eqref{5:37a} is at least
\[
\biggl( \mathop{\int\int_{E \times F }}_{t > s} (
\phi_{\epsilon}*p_{t-s}) (x-y) (\phi_\epsilon*\tilde\kappa)
(x-y) \mu(\d t\, \d x) \mu(\d s\, \d y) \biggr)^2,
\]
which is, in turn, greater than or equal to
%
\begin{equation}
\label{eq:SSS} \qquad\frac{1}4 \biggl( \int\int (\phi_{\epsilon}*p_{|t-s|})
(x-y) (\phi_\epsilon*\tilde\kappa) (x-y) \mu(\d t \,\d x) \mu(\d s\, \d y)
\biggr)^2,
\end{equation}
by symmetry.

The preceding estimates from below the conditional
expectation in \eqref{eq:SS}. And this yields a bound on the right-hand
side of \eqref{eq:SS}. We can also obtain a good estimate
for the left-hand side of \eqref{eq:SS}. Indeed,
the $(N+1)$-parameter filtration $\F$ is commuting; therefore,
according to Cairoli's strong $(2,2)$ inequality
\cite{Kh}, Theorem~2.3.2, page 235,
%
\begin{equation}
\label{eq:Cairoli} \E \Bigl\{ \Bigl( \sup_{(s,\uu)\in\mathbf{Q}^{N+1}_+} \E \bigl[
Z_\epsilon(\mu) \rrvert \F_{s,\uu} \bigr] \Bigr)^2
\Bigr\} \le4^{N+1} \E \bigl( \bigl\llvert Z_\epsilon(\mu )\bigr
\rrvert ^2 \bigr),
\end{equation}
and this is in turn at most a constant times the final quantity in
\eqref{eq:SSS}; compare with \eqref{basic}.
In this way, we are led to the following bound:
%
\begin{eqnarray}
&&\P \bigl\{ (S, \mathbf{U}) \neq\partial \bigr\}
\le\mbox{const}\cdot \biggl[ \int\int(\phi_\epsilon*p_{|t-s|})
(x-y) (\phi_\epsilon*\tilde\kappa) (x-y) \nonumber\\[-8pt]\\[-8pt]
&&\hspace*{198pt}{}\times\mu(\d t \,\d x) \mu(\d s\, \d y)
\biggr]^{-1}.\nonumber
\end{eqnarray}
Since the implied constant is independent of $\epsilon$,
we can let $\epsilon\downarrow0$. As the integrand is lower
semicontinuous, we obtain the following from simple real-variables
considerations:
%
\begin{eqnarray}
&&\P \bigl\{ (S, \mathbf{U}) \neq\partial \bigr\}
\le\mbox{const}\cdot \biggl[ \int\int p_{|t-s|}(x-y) \tilde
\kappa(x-y) \nonumber\\[-8pt]\\[-8pt]
&&\hspace*{141pt}{}\times\mu(\d t \,\d x) \mu(\d s\, \d y) \biggr]^{-1}.\nonumber
\end{eqnarray}
By Proposition~\ref{pr:resolvent}, the term in the reciprocated brackets
is equivalent to the energy $\mathcal{E}_{d-\alpha N} (\mu)$ of $\mu$,
and because $\mu$ is a probability measure on $E \times F$,
we obtain the following:
%
\begin{equation}
\P \bigl\{ (S, \mathbf{U}) \neq\partial \bigr\} \le\mbox{const}\cdot
\mathcal{C}_{d-\alpha N}(E\times F).
\end{equation}
This yields \eqref{eq:goal2}, and hence Theorem~\ref{th:main2}.

\subsection{Proof of Proposition \texorpdfstring{\protect\ref{pr:polar}}{1.4}}
\label{Subsec:polar}

The method for proving Theorem~\ref{th:main2} can be modified to prove
Proposition~\ref{pr:polar}.

\begin{pf*}{Proof of Proposition~\ref{pr:polar} (\normalfont{Sketch})}
The proof for the sufficiency follows a similar line as in Section~\ref{subsec:LB}; we merely exclude all appearances of $X_\alpha(\uu)$,
and keep careful track of the incurred changes. This argument is based
on a
second-moment argument and is standard. Hence, we only give a brief sketch
for the proof of the more interesting necessity.

Assume that $\P\{W(E)\cap F\neq\O\}>0$ and let $\Delta$
be a point that is not in $\R_+$. Define $
\tau:= \inf\{s \in E\dvtx  W(s) \in F\}$ on $\{W(E)\cap F\neq\O\}$,
where $\inf\varnothing:=\Delta$ (in this instance).

Let $\mu$ be the probability measure on $E\times F$ defined by
%
\begin{equation}
\label{eq:mu-2} \mu(G_1\times G_2):= \P \bigl\{ \tau\in
G_1, W(\tau) \in G_2 | \tau\neq\Delta \bigr\}.
\end{equation}
Since $F$ has Lebesgue measure 0, we have $\mu\in\mathcal
{P}_d(E\times F)$.
The rest of the proof is similar to the argument of Section~\ref{subsec:UB}, but
is considerably simpler. Therefore, we omit the many remaining details.
\end{pf*}

\section{Proof of Theorem \texorpdfstring{\protect\ref{th:dimh:ii}}{1.1}}
\label{Sec:Thm-ii}
\setcounter{footnote}{1}
Let us recall Kaufman's uniform dimension result for Brownian motion
\cite{Kaufman69}:
\textit{If $d\ge2$, then outside a single null set $\dim_{_{\mathrm{H}}}W(G)=2\dim_{_{\mathrm{H}}}G$ for all analytic
sets $G\subset\R_+$.} Note that the set $G$ can be random; that is, $G$
can depend on the Brownian path itself.
By considering the random set $G:=W^{-1}(F)$, we can reduce
the proof of Theorem~\ref{th:dimh:ii} to one about determining a
formula for
$\|\dim_{_{\mathrm{H}}}(E \cap W^{-1}(F))\|_{L^\infty(\P)}$;\footnote
{Here is where we
study the case $d\ge2$ separately from the case $d=1$. Kaufman's theorem
fails to hold for one-dimensional Brownian motion. The standard
example is the random set $G:=W^{-1}\{0\}$. For this set,
$\dim_{_{\mathrm{H}}}W(G)=\dim_{_{\mathrm{H}}}\{0\}=0$. And this quantity is
clearly different
from $2\dim_{_{\mathrm{H}}}G$, which is $1$ thanks to a well-known
theorem of Paul L\'evy.} see the paragraph that precedes \eqref{suffice}.

For this purpose, we choose and fix an $\alpha\in(0,1)$, and let
$X_\alpha$ to be a symmetric stable L\'evy process in $\R$ with
index $\alpha$. 
As before, we denote the transition probabilities of $X_\alpha$ by
%
\begin{equation}
g_t(x):= \frac{\P\{ X_\alpha(t)\in\d x\}}{\d x} = \frac{1}{\pi}\int
_0^\infty\cos\bigl(\xi|x|\bigr) \ee^{- t\xi^\alpha/2}\, \d\xi.
\end{equation}

We define $\upsilon$ to be the corresponding
\emph{1-potential density}. That is,
%
\begin{equation}
\upsilon( x):= \int_0^\infty g_t(x)
\ee^{-t} \,\d t.
\end{equation}
It is known that for all $m>0$ there exists
$c_m=c_{m,\alpha}>1$ such that
%
\begin{equation}
\label{upsilon} c_m^{-1}|x|^{\alpha-1} \le\upsilon(x)
\le c_m |x|^{\alpha-1}\qquad  \mbox{if } |x|\le m;
\end{equation}
see \cite{Kh}, Lemma~3.4.1, page 383.
Since $\alpha\in(0,1)$,
the preceding remains valid even when $x=0$, as long
as we recall that $1/0:=\infty$.

For any $\mu\in{\mathcal P}(E\times F)$, the collections
of all probability measures on $E\times F$, and $\beta>0$,
define
%
\begin{equation}
\mathcal{I}_\beta(\mu):= \int\int\frac{\ee^{-\|x-y\|^2/(2|t-s|)}}{
|t-s|^{\beta/2}} \1_{\{s \ne t\}}
\mu(\d s \,\d x) \mu(\d t\, \d y).
\end{equation}

The following forms the first step toward our proof of Theorem~\ref
{th:dimh:ii}.

%
\begin{lemma}\label{lem:main3}
Suppose there exists a $\mu\in{\mathcal P}(E\times F)$ such that
$\mathcal{I}_{d+2(1-\alpha)}(\mu)$ is finite. Then, the random set
$E\cap W^{-1}(F)$ intersects
the closure of $X_\alpha(\R_+)$ with positive probability.
\end{lemma}

%
\begin{remark}\label{rem:other}
It is possible, but significantly
harder, to prove that the sufficient condition of Lemma~\ref{lem:main3}
is also necessary.
We will omit the proof of that theorem, since
we will not need it.
\end{remark}

\begin{pf*}{Proof of Lemma~\ref{lem:main3}} The proof is similar in spirit to that of Proposition~\ref{pr:positiveLeb}.
For all fixed $\epsilon>0$ and probability measures
$\mu$ on $(0,\infty)\times\R^d$, we define
the following parabolic version of \eqref{eq:I},
using the same notation for $\phi_\epsilon:=f_\epsilon*f_\epsilon$, etc.:
%
\begin{equation}
Y_\epsilon(\mu):= \int_0^\infty
\ee^{-t}\, \d t \int\mu(\d s \,\d x) \phi_\epsilon\bigl(W(s)-x\bigr)
\phi_\epsilon\bigl(X_\alpha(t)-s\bigr).
\end{equation}
Just as we did in Lemma~\ref{lem:EZ}, we can find a constant
$c\in(0,\infty)$---depending only on the geometry of $E$ and
$F$---such that
uniformly for all $\mu\in\mathcal{P}(E\times F)$ and $\epsilon\in
(0,1)$,
%
\begin{equation}
\label{EI:inv} \E \bigl[Y_\epsilon(\mu) \bigr] = \int_0^\infty
\ee^{-t} \,\d t \int\mu(\d s\, \d x) (\phi_\epsilon*p_s)
(x) (\phi_\epsilon*g_t) (s) \ge c;
\end{equation}
but now we apply \eqref{upsilon} in place of Lemma~\ref{lem:resolvent}.

And we proceed, just as we did in Proposition~\ref{pr:EZ^2}, and prove that
%
\begin{equation}
\label{main3:L2} \E \bigl(\bigl\llvert Y_\epsilon(\mu)\bigr\rrvert
^2 \bigr) \le\mbox{ const }\cdot \mathcal{I}_{d+2(1-\alpha)}(\mu).
\end{equation}
The only differences between the proof of \eqref{main3:L2}
and that of Proposition~\ref{pr:EZ^2} are the following:
\begin{enumerate}[--]
\item[--] Here we appeal
to Proposition~\ref{pr:isoper:bis}, whereas in Proposition~\ref{pr:EZ^2}
we made use of Proposition~\ref{pr:isoper}; and
\item[--] We apply
\eqref{upsilon} in place of both Proposition~\ref{pr:resolvent}
and Lemma~\ref{lem:resolvent:bis}. Otherwise, the details of
the two computations are essentially the same.
\end{enumerate}

Lemma~\ref{lem:main3} follows from another application of
the Paley--Zygmund lemma \cite{Kh}, page 72, to \eqref{EI:inv}
and \eqref{main3:L2}; the Paley--Zygmund lemma is used in a similar way
as in the proof of the first half of Theorem~\ref{th:main2}. We
omit the details, since this is a standard second-moment computation.
\end{pf*}

Next, we present measure-theoretic conditions
that are respectively sufficient and necessary
for $\mathcal{I}_{d+2(1-\alpha)}(\mu)$ to be
finite for some Borel space--time probability
measure $\mu$ on $E\times F$.

%
\begin{lemma}\label{lem:1}
We always have
%
\begin{equation}
\label{pr:p1} \dim_{_{\mathrm{H}}}(E\times F;\varrho) \le \sup \Bigl\{ \beta>0\dvtx
\inf_{\mu\in\mathcal{P}(E\times F)} \mathcal{I}_\beta(\mu)<\infty \Bigr\}.
\end{equation}
\end{lemma}

\begin{pf}
For all space--time probability measures $\mu$,
and $\tau>0$ define the
\emph{space--time $\tau$-dimensional
Bessel--Riesz energy} of $\mu$ as
%
\begin{equation}
\Upsilon_\tau(\mu;\varrho):= \int\int\frac{\mu(\d s\, \d x)
\mu(\d t\, \d y)}{[\varrho((s,x);(t,y))]^\tau}.
\end{equation}
A suitable formulation of Frostman's theorem \cite{Taylor:TW}
implies that
%
\begin{equation}
\label{eq:frost} \dim_{_{\mathrm{H}}}(E\times F;\varrho) = \sup \bigl\{ \tau>0\dvtx
\Upsilon_\tau(\mu;\varrho)<\infty \bigr\}.
\end{equation}

We can consider separately the cases that
$\|x-t\|^2\le|s-t|$ and $\|x-y\|^2>|s-t|$,
and hence deduce that
%
\begin{equation}
\frac{\ee^{-\|x-y\|^2/(2|t-s|)}}{
|s-t|^\beta} \le\min \biggl( \frac{c}{\|x-y\|^{2\beta}}, \frac{1}{|s-t|^\beta}
\biggr),
\end{equation}
where $c:=\sup_{z>1} z^{2\beta} \ee^{-z/2}$ is finite.
Consequently, $\mathcal{I}_{2\beta}(\mu)
\le c' \Upsilon_{2\beta}(\mu;\varrho)$,
with $c':=\max(c,1)$, and
\eqref{pr:p1} follows from \eqref{eq:frost}.
\end{pf}

\begin{lemma}\label{lem:2}
With probability one,
%
\begin{equation}
\dim_{_{\mathrm{H}}} \bigl( E\cap W^{-1}(F) \bigr) \le
\frac{\dim_{_{\mathrm{H}}}(E\times F;\varrho)
-d}{2}.
\end{equation}
\end{lemma}

\begin{pf}
Choose and fix some $r>0$.
Let $\mathcal{T}(r)$ denote the collection of
all intervals of the form $[t-r^2,t+r^2]$
that are in $[1/q,q]$. Also, let
$\mathcal{S}(r)$ denote the collection of all
closed Euclidean $[\ell^2]$ balls of radius $r$ that are
contained in $[-q,q]^d$.
Recall that $X_\alpha$ is a symmetric
stable process of index $\alpha\in(0,1)$
that is independent of $W$. It is well known
that uniformly for all
$r\in(0,1)$,
%
\begin{equation}
\sup_{I\in\mathcal{T}(r)}\P \bigl\{ X_\alpha\bigl( [0,1]\bigr)
\cap I\neq \O \bigr\}\le\mbox{const}\cdot r^{2(1-\alpha)};
\end{equation}
see \cite{Kh}, Lemma~1.4.3, page 355, for example.
It is just as simple to prove
that the following holds uniformly for all
$r\in(0,1)$:
%
\begin{equation}
\label{Eq:Hitting} \sup_{I\in\mathcal{T}(r)} \sup_{J\in\mathcal{S}(r)} \P
\bigl\{ W(I)\cap J \neq\O \bigr\}\le\mbox{const}\cdot r^d.
\end{equation}
[Indeed, conditional on $\{W(I)\cap J\neq\O\}$,
the random variable $W(t)$ comes to within $r$ of $J$
with a minimum positive probability, where $t$ denotes the smallest
point in $I$.] 
Because $W(I)\cap J\neq\O$ if and only if
$W^{-1}(J)\cap I\neq\O$, it follows that uniformly for all
$r\in(0,1)$,
%
\begin{equation}
\label{3:1} \qquad\sup_{I\in\mathcal{T}(r)}\sup_{J\in\mathcal{S}(r)} \P
\bigl\{ W^{-1}(J)\cap I\cap X_\alpha\bigl([0,1]\bigr) \neq\O
\bigr\} \le\mbox{const}\cdot r^{d+2(1-\alpha)}.
\end{equation}
Define
%
\begin{equation}
\mathcal{R}:= \bigcup_{r\in(0,1)} \bigl\{ I\times J\dvtx  I
\in\mathcal {T}(r)\mbox{ and } J\in\mathcal{S}(r) \bigr\}.
\end{equation}
Thus, $\mathcal{R}$ denotes the collection of all
\textup{``}space--time parabolic rectangles\textup{''} whose $\varrho$-diameter
lies in the interval $(0,1)$.

Suppose $d+2(1-\alpha)>\dim_{_{\mathrm{H}}}(E\times F;\varrho)$.
By the definition of Hausdorff dimension,
and a Vitali-type covering argument (see Mattila
\cite{Mattila}, Theorem~2.8, page 34) for all $\epsilon>0$,
we can find a countable collection $\{E_j\times F_j\}_{j=1}^\infty$
of elements of $\mathcal{R}$ such that:
(i) $\bigcup_{j=1}^\infty(E_j\times F_j)$ contains $E\times F$;
(ii) the $\varrho$-diameter of $E_j\times F_j$ is positive
and less than one (strictly) for all $j\ge1$; and
(iii)~$\sum_{j=1}^\infty| \operatorname{\varrho\mbox{-diam}}(E_j\times F_j)|^{%
d+2(1-\alpha)} \le\epsilon$.
Thanks to \eqref{3:1},
%
\begin{eqnarray}
&&\P \bigl\{ W^{-1}(F)\cap E\cap X_\alpha\bigl([0,1]
\bigr) \neq\O \bigr\} \nonumber\\
&&\qquad \le\sum_{j=1}^\infty \P
\bigl\{ W^{-1}(F_j)\cap E_j\cap
X_\alpha\bigl([0,1]\bigr) \neq\O \bigr\}
\\
&& \qquad \le\mbox{const}\cdot\sum_{j=1}^\infty
\bigl\llvert \operatorname{\varrho\mbox{-diam}}(E_j\times F_j) \bigr
\rrvert ^{d + 2(1-\alpha)} \le\mbox{const}\cdot\epsilon.\nonumber
\end{eqnarray}
Since neither the implied constant
nor the left-most term depend on the value of $\epsilon$,
the preceding shows that
$W^{-1}(F)\cap E\cap X_\alpha([0,1])$ is empty
almost surely.

Now let us recall half of McKean's theorem
\cite{Kh}, Example~2, page 436:
\textit{If $\dim_{_{\mathrm{H}}}(A)>1-\alpha$, then $X_\alpha([0,1])\cap A$
is nonvoid with positive probability.} We apply
McKean's theorem, conditionally, with $A:=W^{-1}(F)\cap E$ to find that
if $d+2(1-\alpha)>\dim_{_{\mathrm{H}}}(E\times F;\varrho)$, then
%
\begin{equation}
\dim_{_{\mathrm{H}}} \bigl( W^{-1}(F)\cap E \bigr) \le1-\alpha \qquad \mbox{almost surely}.
\end{equation}
The preceding is valid almost surely, simultaneously
for all rational values of $1-\alpha$ that are strictly between one
and $\frac{1}2(\dim_{_{\mathrm{H}}}(E\times F;\varrho)-d)$.
Thus, the result follows.
\end{pf}

\begin{pf*}{Proof of Theorem~\ref{th:dimh:ii}}
By the modulus of
continuity of Brownian motion,
there exists a null set off which
$\dim_{_{\mathrm{H}}}W(A) \le2\dim_{_{\mathrm{H}}}A$, simultaneously for
all Borel sets $A\subseteq\R_+$ that might---or
might not---depend on the Brownian path itself.
Since $W(E\cap W^{-1}(F))=W(E)\cap F$,
Lemma~\ref{lem:2} implies that
%
\begin{equation}
\dim_{_{\mathrm{H}}}\bigl(W(E)\cap F\bigr) \le\dim_{_{\mathrm{H}}}(E\times F;
\varrho) -d \qquad \mbox{almost surely}.
\end{equation}

For the remainder of the proof, we assume that $d\ge2$,
and propose to prove that
%
\begin{equation}
\bigl\llVert \dim_{_{\mathrm{H}}}\bigl(W(E)\cap F\bigr) \bigr\rrVert
_{L^\infty(\P)} \ge\dim_{_{\mathrm{H}}}(E\times F;\varrho) -d.
\end{equation}
Henceforth, we assume without loss of generality that
%
\begin{equation}
\label{assumeTW} \dim_{_{\mathrm{H}}}(E\times F;\varrho) >d;
\end{equation}
for there is nothing left to prove
otherwise. In accord with the theory of Taylor and Watson
\cite{Taylor:TW}, \eqref{assumeTW} implies that
$\P\{W(E)\cap F\neq\O\}>0$.

According to Kaufman's uniform-dimension
theorem \cite{Kaufman69},
the Hausdorff dimension of $W(E)\cap F$ is almost surely
equal to twice the Hausdorff dimension of
$E\cap W^{-1}(F)$. Therefore, it suffices to prove the
following in the case that $d\ge2$:
%
\begin{equation}
\label{suffice} \bigl\llVert \dim_{_{\mathrm{H}}} \bigl(E\cap W^{-1}(F)
\bigr) \bigr\rrVert _{L^\infty(\P)} \ge \frac{\dim_{_{\mathrm{H}}}(E\times F;\varrho)
-d}{2},
\end{equation}
as long as the right-hand side is positive.
If $\alpha\in(0,1)$ satisfies
%
\begin{equation}
\label{suffice1} 1-\alpha< \frac{\dim_{_{\mathrm{H}}}(E\times F;\varrho)
-d}{2},
\end{equation}
than Lemma~\ref{lem:1} implies that
$\mathcal{I}_{d+2(1-\alpha)}(\mu)<\infty$
for some $\mu\in\mathcal{P}(E\times F)$.\vspace*{2pt}
Thanks to Lemma~\ref{lem:main3},
$E\cap W^{-1}(F)\cap\overline{X_\alpha([0,1])}
\neq\O$ with positive probability.
Consequently,
%
\begin{equation}
\label{this} \P \bigl\{ \dim_{_{\mathrm{H}}}\bigl(E\cap W^{-1}(F)
\bigr) \ge1-\alpha \bigr\}>0,
\end{equation}
because the second half of McKean's theorem implies
that \emph{if $\dim_{_{\mathrm{H}}}(A)<1-\alpha$, then $\overline{
X_\alpha(\R_+)}\cap A=\O$ almost surely.}
Since \eqref{this} holds for all $\alpha\in(0,1)$ that
satisfy \eqref{suffice1}, \eqref{suffice} follows. This
completes the proof.
\end{pf*}

%
\begin{remark}
Let us mention the following byproduct of
our proof of Theorem~\ref{th:dimh:ii}: For every
$d\ge1$,
%
\begin{equation}
\bigl\llVert \dim_{_{\mathrm{H}}} \bigl( E\cap W^{-1}(F) \bigr)\bigr
\rrVert _{L^\infty(\P)} = \frac{\dim_{_{\mathrm{H}}} ( E\times F;\varrho )- d}{2}.
\end{equation}
When $d=1$, this was found first by
Kaufman \cite{Kaufman72}, who used other arguments
(for the harder half).
See Hawkes \cite{Hawkes78}
for similar results in case $W$ is replaced by
a stable subordinator of index $\alpha\in(0,1)$.
\end{remark}

We conclude this paper with some problems that continue to elude us.

\begin{OP*}
Theorems \ref{th:dimh:ii} and \ref{th:dimh} together
imply that when $d\ge2$ and $F \subset\R^d$ has Lebesgue measure 0,
%
\begin{equation}
\label{eq:OP} \sup \Bigl\{\gamma>0\dvtx  \inf_{\mu\in\mathcal{P}_d (E\times F)}
\mathcal{E}_\gamma(\mu)<\infty \Bigr\} = \dim_{_{\mathrm{H}}}(E\times F;
\rho)-d.
\end{equation}
The preceding is a kind of \textup{``}parabolic Frostman theorem.\textup{''}
And we saw in the \hyperref[sec1]{Introduction} that \eqref{eq:OP} is
in general false when $d=1$. We would like to
better understand why the one-dimensional case is so different
from the case $d\ge2$. Thus, we are led naturally to
a number of questions, three of which we state below:
\begin{enumerate}[P3.]
\item[P1.] Equation \eqref{eq:OP} is, by itself,
a theorem of geometric measure theory.
Therefore, we ask, \textup{``}\textit{Is there a direct proof of
\textup{\eqref{eq:OP}} that does not involve random processes}, \textit{broadly
speaking}, \textit{and
Kaufman}'\textit{s uniform-dimension
theorem \cite{Kaufman69}}, \textit{in particular}\textup{''}?
\item[P2.] When $d\ge2$, \eqref{eq:OP} gives an interpretation
of the
capacity form on the left-hand side of \eqref{eq:OP}
in terms of the geometric object on
the right-hand side. Can we understand the left-hand
side of \eqref{eq:OP} geometrically in the case that $d=1$?
\item[P3.] The following interesting question is due to an anonymous
referee: Are there quantitative relationships between a rough hitting-type
probability of the form $\P\{\dim_{_{\mathrm{H}}}(W(E)\cap F)>\gamma\}$ and
the new capacity form of Benjamini et al. \cite{BPP97} (see also
\cite{MoertersPeres}, Theorem~8.24)?
We suspect the answer is \textup{``}yes,\textup{''} but do not have a proof.
\end{enumerate}
\end{OP*}

\section*{Acknowledgments}
Many hearty thanks are due to
Professors Gregory Lawler and Yuval Peres. The former
showed us the counterexample in the \hyperref[sec1]{Introduction}, and the latter
introduced us to the problem that is being considered here.

We thank the anonymous referee for pointing out several mistakes in an
earlier formulation of Theorem~\ref{th:dimh} in a previous draft of this
manuscript.
%

%



\printaddresses

\end{document}